\documentclass{amsart}
\usepackage[utf8]{inputenc}
\usepackage{amssymb}

\usepackage{color}
\usepackage{enumitem}

\usepackage[colorlinks=true, pdfstartview=FitV, linkcolor=blue, pagebackref, citecolor=blue, urlcolor=blue]{hyperref}

\makeatletter
\@namedef{subjclassname@2020}{\textup{2020} Mathematics Subject Classification}
\makeatother

\numberwithin{equation}{section}
\theoremstyle{plain}

\newtheorem{lem}[equation]{Lemma}
\newtheorem{cor}[equation]{Corollary}
\newtheorem{prop}[equation]{Proposition}

\theoremstyle{definition}
\newtheorem{defn}[equation]{Definition}
\newtheorem{eg}[equation]{Example}

\theoremstyle{remark}
\newtheorem{rmk}[equation]{Remark}

\DeclareMathOperator{\Sym}{S}
\newcommand{\Herm}{\mathcal{H}}
\DeclareMathOperator{\im}{im}
\DeclareMathOperator{\ad}{ad}
\DeclareMathOperator{\Ad}{Ad}
\DeclareMathOperator{\rank}{rank}

\DeclareMathOperator{\Unit}{\mathcal{U}}

\newcommand{\Gm}{\mathbb{G}_{\mathrm{m}}}

\newcommand{\iso}{\xrightarrow{\sim}}

\newcommand{\slot}{\overline{\ }}

\newcommand{\la}{\lambda}
\newcommand{\s}{\sigma}

\newcommand{\kalg}{\overline{k}}

\DeclareMathOperator{\Proj}{proj}

\newcommand{\ot}{\otimes} % tensor product

\newcommand{\hv}{h^\vee} % dual Coxeter number

\newcommand{\am}{\mathbin{\diamond}} % multiplication in A(\g)
\newcommand{\jord}{\mathbin{\bullet}} % Jordan product
\newcommand{\Mm}{\mathbin{\ast}}

\newcommand{\g}{\mathfrak{g}}
\newcommand{\h}{\mathfrak{h}}

\renewcommand{\sl}{\mathfrak{sl}}
\newcommand{\gl}{\mathfrak{gl}}
\newcommand{\lsub}{\mathfrak{l}}

\newcommand{\GL}{\mathrm{GL}}
\newcommand{\SO}{\mathrm{SO}}
\newcommand{\Sp}{\mathrm{Sp}}
\newcommand{\SL}{\mathrm{SL}}

\newcommand{\bil}[2]{\langle #1 | #2 \rangle}
\newcommand{\R}{\mathbb{R}}
\newcommand{\C}{\mathbb{C}}
\newcommand{\F}{\mathbb{F}}
\newcommand{\Z}{\mathbb{Z}}
\newcommand{\Q}{\mathbb{Q}}

\newcommand{\eo}{e_{\otimes}}

\newcommand{\es}{e_{\mathrm{S}}}

\newcommand{\sumroots}{\delta}

\newcommand{\tbil}{\tau}

\newcommand{\counit}{\varepsilon}

\DeclareMathOperator{\car}{char}
\DeclareMathOperator{\Lie}{Lie}
\DeclareMathOperator{\Tr}{Tr}

\DeclareMathOperator{\End}{End}

\DeclareMathOperator{\Aut}{Aut}
\DeclareMathOperator{\Id}{Id}

\newcommand{\hst}{{\widetilde{\alpha}}} % highest root

\title[A class of continuous non-associative algebras]{A class of continuous non-associative algebras arising from algebraic groups including $E_8$}
\author{Maurice Chayet}

\author{Skip Garibaldi}

\begin{document}

\begin{abstract}
We give a construction that takes a simple linear algebraic group $G$ over a field and produces a commutative, unital, and simple non-associative algebra $A$ over that field.  Two attractions of this construction are that (1) when $G$ has type $E_8$, the algebra $A$ is obtained by adjoining a unit to the 3875-dimensional representation and  (2) it is effective, in that the product operation on $A$ can be implemented on a computer.  A description of the  algebra in the $E_8$ case has been requested for some time, and interest has been increased by the recent proof that $E_8$ is the full automorphism group of that algebra.  The algebras obtained by our construction have an unusual Peirce spectrum. 
\end{abstract}

\subjclass[2020]{Primary 17B25; Secondary 17D99, 20G41}

\maketitle

\setcounter{tocdepth}{1}
\tableofcontents

%%%%%%%%%%%%%%%%%%%%%%%%%%%%%%%%%%%%%%%%%%%%%%%%%%%%%%%%%
\section{Introduction}

We present a construction that takes an absolutely simple linear algebraic group $G$ over a field $k$ and produces a commutative, unital non-associative algebra that we denote by $A(\g)$.  As a vector space, $A(\g)$ is a subspace of the symmetric square $\Sym^2 \g$ of the Lie algebra $\g$ of $G$.  We give an explicit formula \eqref{prod.full} for the product on $A(\g)$, which makes our construction effective in the sense that one can perform computer calculations (\S\ref{final.sec}), although we do not rely on computer calculations for our results. 
There is a natural symmetric bilinear form on $A(\g)$, which we show is associative (\S\ref{assoc.sec}) and nondegenerate (\S\ref{ndeg.sec})  and positive-definite in case $k = \R$ and $G$ is compact.  We leverage this and the structure of $A(\g)$ as a representation of $G$ to show that it is a simple $k$-algebra (Cor.~\ref{metrized.simple}).

This work may be viewed in the context of the general problem of describing exceptional groups as automorphism groups, which dates back to Killing's 1889 paper \cite{Killing2}.  As an example, the Lie group $G_2$ can be viewed as the automorphism group of the octonions (E.~Cartan \cite{Cartan:real}), the stabilizer of a cross product on $\R^7$ (F.~Engel, \cite{Engel}, \cite{Herz:alt}), or the symmetry group for a ball of radius 1 rolling on a fixed ball of radius 3 without slipping or twisting (E.~Cartan, \cite{BaezH:rolling}).  For $E_8$, it is known from \cite{GG:simple} that it is the identity component of the stabilizer of an octic form on the Lie algebra $\mathfrak{e}_8$ and that it is the automorphism group of the $E_8$-invariant algebra on its 3875-dimensional irreducible representation.  (See also \cite[\S3]{G:e8} or \cite[\S16]{GG:simple} for broader discussions of other realizations.)  The latter description of $E_8$ is known to be true even though this algebra is not well-understood; this paper gives explicit and effective formulas for calculating in the algebra.  We note here that $\Aut(A(\mathfrak{e}_8)) = E_8$, see Prop.~\ref{irred.prop}.

The algebras $A(\g)$ constructed here are ``non-generic'' in the sense of \cite{KrasnovTkachev}, meaning that $A(\g) \otimes \kalg$ contains infinitely many idempotents, for $\kalg$ an algebraic closure of $k$.  Moreover,
the Peirce spectrum of $A(\g) \otimes \kalg$, i.e., the union of the set of eigenvalues for left multiplication by $u$ as $u$ varies over idempotents of $A(\g) \otimes \kalg$, is infinite, see Example \ref{cont}.  In case $k = \R$  and apart from types $A_1$ and $A_2$, this collection of eigenvalues includes the unit interval, and consequently one might call these algebras ``continuous'' as we have done in the title of the article.  We remark that this kind of situation --- where a popular property holds for generic cases but fails for a structure naturally associated with a simple algebraic group $G$ --- is familiar from the study of homogeneous $G$-invariant polynomials.  In that setting, a generic homogeneous polynomial is non-singular, yet $G$-invariant polynomials of degree $\ge 3$ are singular \cite{OS}, such as the determinant on $n$-by-$n$ matrices.

Ignoring some very small cases, the algebras $A(\g)$ are not power-associative.  This is not a defect of our construction.  We show that even if one alters the one choice we made in the construction, the resulting algebra would still not be power-associative, see Prop.~\ref{nonpa}\eqref{nonpa.nonpa} and Remark \ref{pa.crazy}.

In the penultimate section, \S\ref{const2}, we give an alternative realization of $A(\g)$ inside $\End(V)$ where $V$ is the natural module for $G$ of type $A_2$, $G_2$, $F_4$, $E_6$, or $E_7$.  We use this alternative realization to explicitly compute $A(\sl_3)$ (Example \ref{A2}).  We conclude with an appendix (\S\ref{unitizing.sec}) giving various results about adding a unit to a non-associative algebra that we refer to in the body of the paper.

\smallskip

We work over a rather general field $k$ and do not assume that $G$ is split, although our results are new already in the case where $k$ is the complex numbers $\C$.  The additional generality comes at hardly any cost due to the tools we use.  Readers who are not interested in the full generality are invited to assume throughout that $k = \C$ and identify the symbols $H^0(\la) = V(\la) = L(\la)$.

An unusual feature of our work is that the case where $G$ is of type $E_8$ is less complicated than other $G$ in several ways, at least when $k = \C$.  For $E_8$, one has extra formulas to use, such as Okubo's Identity $\Tr(\pi(X)^4) = \alpha_\pi K(X,X)^4$ (Lemma \ref{2.4a}, which holds for all $G$ of exceptional type) and a similar identity for $\Tr(\pi(X)^6)$ (which holds for type $E_8$).    Another way that $E_8$ is less complicated is that the Molien series $1 + t^2 + t^3 + 3t^4 + 3t^5 + 10t^6 + 16t^7 +  \cdots$ for $E_8$ acting on its 3875-dimensional representation $V$ has coefficients no greater than the Molien series for the corresponding representations of other groups of type $E$, $F$, or $G$.  Yet another way is that the second symmetric power $\Sym^2 V$ is a sum of 6 irreducible terms, which is minimal among the types $E$, $F$, and $G$.  

Our original approach to the material in this paper was to focus on the case of $E_8$ and leverage these tools.  In this way, we discovered the product formula on $A(\g)$, and only in hindsight did we see that it was a general construction that worked for all simple $G$.
Due to this inverted approach, preparing this document took more than three years.  Just before we intended to release this work on the arxiv, the paper \cite{DeMedtsVC} appeared, which studies algebras that are almost the same, albeit restricted to the cases where the root system of $G$ is simply laced and $G$ is split and $\car k = 0$, see Remark \ref{DMVC.rmk} below.  Both that article and this one view the algebras as subspaces of $\Sym^2 \g$ and provide an associative symmetric bilinear form (we say $A(\g)$ is metrized, whereas they say Frobenius), but from there our approaches and results diverge.

%\smallskip
%{\small \subsection*{Acknowledgements.} We thank Robert Guralnick and Benedict Gross for helpful conversations, which improved some of the proofs, and anonymous referees for their comments, which improved the exposition.}

%%%%%%%%%%%%%%%%%%%%%%%%%%%%%%%%%%%%%%%%%%%%%%%%%%%%%%%%%
\section{Background material} \label{gen.sec}

Let $k$ be a field of characteristic different from 2 and suppose that $\g$ is a Lie algebra over $k$ whose Killing form, $K$, is nondegenerate.  Then the $k$-algebra  of linear transformations of $\g$, denoted $\End(\g)$, has a  ``transpose'' operator $\top$ given by 
\[ % \begin{equation} \label{1.3}
K(T(X), Y) = K(X, T^\top(Y)) \quad \text{for $T \in \End(\g)$ and $X, Y \in \g$.}
\] %\end{equation}

\subsection*{Identification of representations}
Another way to view the nondegeneracy of $K$ is that it provides a $\g$-equivariant isomorphism of $\g$-representations
\begin{equation} \label{dual.id}
\g \iso \g^* \quad \text{via} \quad X \mapsto K(X,\slot).
\end{equation}
This identification extends to an isomorphism of $\g$-modules 
\begin{equation} \label{end.id}
\g \ot \g \iso \g \ot \g^* = \End(\g).
\end{equation}

As $\car k \ne 2$, the natural surjection of $\g \ot \g$ onto the 2nd symmetric power $\Sym^2 \g$ is split by the map
\begin{equation} \label{sym.emb}
\Sym^2 \g \hookrightarrow \g \ot \g \quad \text{given by\ } XY \mapsto \frac12 (X \ot Y + Y \ot X).
\end{equation}

\begin{defn} \label{P.def}
Define $P \!: \Sym^2 \g \hookrightarrow \End(\g)$ as the composition of \eqref{end.id} with \eqref{sym.emb}.  It is $\g$-equivariant and its image is the space 
\[
\Herm(\g) := \{ T \in \End(\g) \mid T^\top = T \}
\]
of symmetric operators.  We have:
\[
P(XY) = \frac12 \left[ X \ot K(Y, \slot) + Y \ot K(X, \slot) \right] \quad \text{for $X, Y \in \g$.}
\]
\end{defn}

\begin{eg} \label{es.eg}
For  $\{ X_i \}$ a basis of $\g$ and $\{ Y_i \}$ the dual basis with respect to $K$, set $\eo := \sum X_i \ot Y_i \in \g \ot \g$ and $\es := \sum X_i Y_i$, the image of $\eo$ in $\Sym^2 \g$.  Neither $\eo$ nor $\es$ depend on the choice of the $X_i$'s.  Moreover, the identification \eqref{end.id}
sends $\eo \mapsto \Id_\g$, so $P(\es) = \Id_\g$.
\end{eg}

The spaces $\End(\g)$ and $\Herm(\g)$ are Jordan algebras under the 
Jordan product $\jord$ defined by
\begin{equation}  \label{jord.def}
T \jord U := \frac12 (TU + UT)   \quad \text{for $T, U \in \End(\g)$}.
\end{equation}

\begin{eg} \label{jord.eg}
For $X_1, X_2, X_3, X_4 \in \g$, we have:
\begin{multline*}
P(X_1 X_2) \jord P(X_3 X_4) = \frac14 \left[ K(X_1, X_3) P(X_2 X_4) + K(X_1, X_4) P(X_2 X_3)\right. \\
\left. + K(X_2, X_3) P(X_1 X_4) + K(X_2, X_4) P(X_1 X_3) \right].
\end{multline*}
Therefore, for any subspace $\lsub$ of $\g$, $P(\Sym^2 \lsub)$ is a Jordan subalgebra of $\Herm(\g)$.
If $K(X,X) \ne 0$, then the element $P(X^2)/K(X, X)$ is an idempotent in the Jordan algebra.   

Suppose that furthermore $\lsub$ has an orthonormal basis $X_1, \ldots, X_r$.  Then for $i \ne j$, $P(X_i^2) \jord P(X_j^2) = 0$ and $P(X_i^2) \jord P(X_i X_j) = \frac12 P(X_i X_j)$.  In particular, $\sum P(X_i^2)$ is the identity element in $P(\Sym^2 \lsub)$.
\end{eg}

\subsection*{Global hypotheses}  We now add hypotheses that will be assumed until the start of Appendix \ref{unitizing.sec}.  We will assume that \emph{$\g$ is the Lie algebra of an absolutely simple linear algebraic group $G$ over $k$.}  That is, $G$ is a smooth affine group scheme of finite type over $k$, and $G \times \kalg$ is simple, i.e., $G \times \kalg$ is connected, semisimple (= has trivial radical), is $\ne 1$, and its associated root system is irreducible. 

We write $h$ for the Coxeter number and $\hv$ for the dual Coxeter number of (the root system of) $G$; some examples are given in Table \ref{dual.coxeter} below. It is true that $\rank G < \hv \le h$, and the root system of $G$ is simply laced if and only if $\hv = h$.

We additionally assume until the start of the appendix that 
\emph{$\car k$ is zero or at least $h + 2$.}
Consequently: The integers 2, $\rank G$, $\hv$, $\hv + 1$ are not zero in $k$, so the same is true for $\dim G = (\rank G)(h + 1)$.  Examining type of $G$ in turn, we find: (1) The characteristic is ``very good'' for $G$.  
(2) The determinant of the Cartan matrix is not zero in $k$.  
(3) The ratio $\nu_G$ of the square-length of a long root to that of a short root (equivalently, the valence of the Dynkin diagram of $G$) is not zero in $k$.  

The discriminant of the Killing form $K$ on $\g$ can be expressed as a product of integers we have already observed are not zero in $k$ \cite[p.~E-14, I.4.8(a)]{SpSt}, and therefore $K$ is nondegenerate.  Finally, $\g$ is a simple Lie algebra that is an irreducible representation of $G$ \cite{Hiss}; it follows that, if $G'$ is isogenous to $G$, then $\g' \cong \g$.

\subsection*{Representations}
Suppose $G$ is split and put $\h$ for the Lie algebra of a split maximal torus $T$.  For a dominant weight $\la \in T^*$, we write $L(\la)$ for the irreducible representation of $G$ with highest weight $\la$.  The dimension and character of $L(\la)$ may depend on the characteristic of $k$ and not just on root system data.  However, there are representations $H^0(\la)$ and $V(\la)$ of $G$, both with highest weight $\la$, which equal $L(\la)$ when $\car k$ is zero or ``big enough'' (where what counts as big enough depends on $G$ and $\la$), and whose character is the same as the character of the irreducible representation over $\C$ with highest weight $\la$.  
The representations $V(\la)$ are called \emph{Weyl modules}; a basic example of such is the tautological representation of $\SO_n$.
See \cite{Jantzen} for background on these representations.
We use the fact that these representations are defined over $\Z$, see \cite[II.8.3]{Jantzen}.  See also \S\ref{rep.sec} for a discussion of the case where $G$ is not assumed to be split.

\subsection*{Casimir operator}
Put $\bil{\ }{\ }$ for the canonical bilinear form on the weight lattice of $G$, as defined in \cite[\S{VI.1.12}]{Bou:g4} or \cite[p.~115]{Dynk:ssub}; it is the unique nonzero and Weyl-group-invariant inner product satisfying $\bil{\lambda}{\lambda'} = \sum_\alpha \bil{\lambda}{\alpha} \bil{\lambda'}{\alpha}$, where $\alpha$ varies over the roots.  Then $\bil{\alpha}{\alpha} = 1/\hv$ for every long root $\alpha$, see \cite[p.~150]{Suter}. 

More generally, for each root $\alpha$, define $\nu_\alpha := 1$ if $\alpha$ is long and $\nu_\alpha := \nu_G$ if $\alpha$ is short.  By definition, then, $\bil{\alpha}{\alpha} = (\nu_\alpha \hv)^{-1}$ for every root $\alpha$, and this is not zero in $k$.

For the next two lemmas, we set $R := \Z_{(\car k)}$, the subring of $\Q$ whose nonzero elements are the fractions with denominator not divisible by $\car k$.  Note that $R$ is a local ring, $R/(\car k) \subseteq k$,  and $R = \Q$ if $\car k = 0$.

\begin{lem} \label{iprod.R}
For weights $\la, \la'$, the element $\bil{\la}{\la'}$ belongs to $R$.
\end{lem}

\begin{proof}
It suffices to find a $c \in R^\times$ so that $\bil{\la}{c\la'}$ is in $R$.  If $\la'$ is a root, we take $c := 2/\bil{\la'}{\la'} = 2\nu_{\la'} \hv$.  Because $\la'$ is a root, $\bil{\la}{c\la'}$ is an integer, so in $R$.

If $\la'$ is in the root lattice, then the conclusion follows from the previous case by bilinearity.

For general $\la'$, we take $c := h$.  Since $h\la'$ is in the root lattice, $\bil{\la}{h\la'}$ is in $R$.
\end{proof}

We put $\sumroots$ for the sum of the positive roots.

\begin{lem} \label{qcas}
Suppose that the representation $\pi \!: G \to \GL(V)$ is equivalent to $H^0(\la)$ or $V(\la)$ over the algebraic closure of $k$ for some dominant weight $\la$.  Then:
\begin{enumerate}
\item  \label{qcas.end} For $\{ X_i \}$ a basis of $\g$ and $\{ Y_i \}$ the dual basis with respect to $K$, $\sum \pi(X_i) \, \pi(Y_i) = \bil{\la}{\la + \sumroots} \Id_V$ where $\sumroots$ is  the sum of the positive roots.
\item \label{qcas.tr} For all $x, y \in \g$ we have 
\[
\Tr (\pi(x)\, \pi(y)) = \frac{\bil{\la}{\la+\sumroots} \dim V}{\dim G} K(x,y).
\]
\end{enumerate}
\end{lem}

In the statement, we have abused notation by writing $\pi$ also for the differential $\g \to \gl(V)$ of $\pi$.

\begin{proof}[Sketch of proof]
In case $k$ is algebraically closed of characteristic zero, this result is about an irreducible representation and the claims are part of the 
usual theory of the quadratic Casimir operator $\sum X_i Y_i \in U(\g)$ as in, for example, \cite[\S{VIII.6.4}, Cor.]{Bou:g7} or \cite[Th.~2.5]{Dynk:ssub}. 

In case $\car k = 0$, it suffices to verify the claims over an extension field, for which we take the algebraic closure of $k$.

Now suppose that $\car k$ is a prime $p$ and $G$ is split.  There is a split group $G_R$  and representation $\pi_R$, both defined over $R$, whose base change to $k$ is equivalent to $G$, $\pi$.  As $\bil{\la}{\la + \sumroots}$ and $(\dim G)^{-1}$ are in $R$, the claims amount to certain polynomials over $R$ being zero. Those polynomials are zero over the field of fractions $\Q$ of $R$, so they are also zero over the quotient field $\F_p$ and therefore over $k$.

Finally, if $\car k$ is prime, again it suffices to verify the claims over the algebraic closure of $k$, where $G$ is split.
\end{proof}

%%%%%%%%%%%%%%%%%%%%%%%%%%%%%%%%%%%%%%%%%%%%%%%%%%%%%%%%%
\section{The representation $A(\g)$} \label{Ag.sec}

Define a map 
$\g \otimes \g \to \End(\g)$ via $X \ot Y \mapsto \hv (\ad X)(\ad Y) + XK(Y, \slot)$.  It is bilinear, so provides a $G$-equivariant linear map $\g \ot \g \to \End(\g)$.  Composing this with \eqref{sym.emb}, we find a $G$-equivariant linear map $S \!: \Sym^2(\g) \to \End(\g)$ such that
\begin{equation} \label{S.def}
S(XY) := \hv \ad(X) \jord \ad(Y) + P(XY),
\end{equation}
where $P$ is as in Definition \ref{P.def} and $\jord$ denotes the Jordan product  \eqref{jord.def}.

Since $(\ad X)^\top = -\ad X$ for all $X \in \g$, we find that $S(XY)$ belongs to $\Herm(\g)$. 
Since $S$ is linear in $X$ and in $Y$ and symmetric in the two terms, it extends linearly to all of $\Sym^2(\g)$.    We set:
\begin{equation} \label{Ag.def}
{A(\g) := \im S \quad \subseteq \Herm(\g).}
\end{equation}

\begin{eg} \label{tr.Sx}
For $X \in \g$ we have
\[
\Tr(S(X^2)) = \hv K(X,X) + \Tr (X K(X, \slot)) = (\hv + 1)K(X,X).
\]
Linearizing this shows that $\Tr(S(XY)) = (\hv + 1)K(X,Y)$ for $X, Y \in \g$.
\end{eg}

\begin{eg} \label{S.es}
For $S(\es)$, we have $P(\es) = \Id_\g$ as in Example \ref{es.eg} and $\sum (\ad X_i)(\ad Y_i) = \Id_\g$ as in Lemma \ref{qcas}\eqref{qcas.end}.  Therefore, $S(\es) = (\hv + 1) \Id_\g$. 
\end{eg}

\subsection*{The split case}
Suppose that $G$ is split, i.e., contains a split maximal torus $T$ defined over $k$.  (This is automatic if $k$ is algebraically closed.)  Fix a Chevalley basis of $\g$ with respect to $\h := \Lie(T)$ in the sense of \cite{St}, \cite{SpSt}, or \cite[\S{XX.2.11}]{SGA3.3:new}.  That is, for each root $\alpha$, define elements $H_\alpha \in \h$ and $X_\alpha \in \g$ so that $X_\alpha$ spans the $\alpha$ weight space (for the action of $T$ on $\g$), $\g = \h \oplus \bigoplus_\alpha kX_\alpha$,
\[
[H_\beta, X_\alpha ]  = \beta^\vee(\alpha) X_\alpha\quad \text{and} \quad [X_\alpha, X_{-\alpha}] = H_\alpha.
\]
(This last equation differs by a sign from the one used in \cite[\S{VIII.2.2}]{Bou:g7}.)
 We note that for any root $\alpha$, 
 \begin{equation} \label{K.spst}
 K(X_\alpha, X_{-\alpha}) = K(H_\alpha, H_\alpha)/2 = 2 \nu_\alpha \hv
 \end{equation}
by the formulas in  \cite[pp.~E-14, E-15]{SpSt}.

\begin{lem} \label{high.wt}
Maintain the notation of the preceding paragraph.  Suppose that $\alpha$ and $\beta$ are roots of $G$ such that $\alpha + \beta$ is not a root.
\begin{enumerate}
\item \label{high.0} If $\bil{\alpha}{\beta} = 0$, then $S(X_\alpha X_\beta) \ne 0$ in $A(\g)$.
\item \label{high.plus} Suppose $\bil{\alpha}{\beta} > 0$.  Then $S(X_\alpha X_\beta) \ne 0$ in $A(\g)$ if and only if there are two root lengths and $\alpha$ and $\beta$ are both short.
\end{enumerate}
\end{lem}

\begin{proof}
Since $X_\alpha$, $X_\beta$ commute in $\g$, so do $\ad X_\alpha$, $\ad X_\beta$ in $\End(\g)$.  Therefore,
\begin{equation} \label{high.key}
S(X_\alpha X_\beta) X_{-\alpha} = (\nu_\alpha - \alpha^\vee(\beta)) \hv X_\beta + \frac12 K(X_\beta, X_{-\alpha}) X_\alpha.
\end{equation}

If $\bil{\alpha}{\beta} = 0$, then the only nonzero term on the right side of \eqref{high.key} is $\nu_\alpha \hv X_\beta \ne 0$, verifying \eqref{high.0}.

We now prove:
\begin{enumerate}[label=(\arabic*$'$)]
\setcounter{enumi}{1}
\item \label{high.plus2} \emph{Suppose $\bil{\alpha}{\beta} > 0$.  Then $S(X_\alpha X_\beta) X_{-\alpha} \ne 0$ in $\g$ if and only if $\alpha$ and $\beta$ are both not long.}
\end{enumerate}
If $\alpha = \beta$, then $\alpha^\vee(\beta) = 2$ and \eqref{high.key}  equals $2(\nu_\alpha - 1)\hv X_\beta$.  This is nonzero if and only if $\alpha$ is not long, verifying \ref{high.plus2} in this case.

If $\alpha \ne \beta$, then \eqref{high.key} equals $ (\nu_\alpha - \alpha^\vee(\beta))\hv X_\beta$.   If $\alpha$ is not long, then either (a) $\beta$ is long, $\nu_\alpha = \alpha^\vee(\beta)$, and \eqref{high.key} is  zero or (b) $\beta$ is also not long, $\alpha^\vee(\beta) = 1$, and \eqref{high.key} is not zero.  If $\alpha$ is long, then $\nu_\alpha = \alpha^\vee(\beta) = 1$, see for example \cite[\S{VI.1.3}]{Bou:g4}.  This completes the verification of \ref{high.plus2}.

To complete the proof of the lemma, we assume that $\bil{\alpha}{\beta} > 0$ and at least one of $\alpha$, $\beta$ is long, and verify that $S(X_\alpha X_\beta) = 0$.  Because $S(X_\alpha X_\beta) H = 0$ for all $H \in \h$, it remains to evaluate 
\begin{equation} \label{high.g}
S(X_\alpha X_\beta) X_{-\gamma} = \hv [X_\beta, [X_\alpha, X_{-\gamma}]]\quad \text{for $\gamma \ne \alpha,\beta$.}
\end{equation}
By symmetry, we may assume that $\alpha$ is long, so in the Weyl orbit of the highest root $\hst$, and we may even assume that $\alpha = \hst$.  If any of $\hst - \gamma$, $\beta - \gamma$, or $\hst+\beta-\gamma$ is not a root, then \eqref{high.g} is zero, as claimed.  

For sake of contradiction, suppose that all three are roots.  This implies $\beta \ne \hst$, for otherwise $\hst + \beta - \gamma = 2\hst - \gamma$ is a root, whence $\gamma = \hst$, a contradiction.   Since $\gamma$ and $\hst - \gamma$ are roots, $\gamma$ is positive. 

Note that if $\rho$ is any root orthogonal to $\hst$, then since at least one of $\hst \pm \rho$ is not a root, neither can be.  Consequently, $\bil{\hst}{\gamma} \ne 0$.  It follows that $\bil{\hst}{\gamma} > 0$, since $\hst + \gamma$ is not a root and $\hst \ne -\gamma$.  Now $\hst$ is long and $\bil{\hst}{\beta}$, $\bil{\hst}{\gamma}$ are positive, so $\hst^\vee(\beta) = \hst^\vee(\gamma) = 1$, whence $\bil{\hst}{\beta - \gamma} = 0$, contradicting the hypothesis that $\hst + \beta - \gamma$ is a root.
\end{proof}

\begin{cor} \label{high.2hst}
$2\hst$ is not a weight of $A(\g)$.
\end{cor}

\begin{proof}
The $2\hst$ weight space in $\Sym^2 \g$ is spanned by $X_\hst^2$, yet $S(X_\hst^2) = 0$ by Lemma \ref{high.wt}\eqref{high.plus}.
\end{proof}

%%%%%%%%%%%%%%%%%%%%%%%%%%%%%%%%%%%%%%%%%%%%%%%%%%%%%%%%%
\section{The commutative algebra $A(\g)$}

Recall the vector space $A(\g)$ defined in \eqref{Ag.def}.
Define, for $A, B, C, D \in \g$:
\begin{align}
S(AB)\am S(CD) = &\ \frac{\hv}2 \left( S(A, (\ad C \jord \ad D) B) + S((\ad C \jord \ad D) A, B) \right) \notag \\
&+\frac{\hv}2 \left( S( C, (\ad A \jord \ad B) D) + S((\ad A \jord \ad B) C, D) \right) \notag \\
&+ \frac{\hv}2 \left( S([A,C],[B,D]) + S([A,D],[B,C]) \right) \label{prod.full} \\
&+ \frac14 \left( K(A,C)S(B,D)+ K(A,D)S(B,C) \right) \notag \\
&+ \frac14 \left( K(B,C)S(A,D) + K(B,D)S(A,C) \right)\notag
\end{align}
in $A(\g)$, where on the right side we have added extra commas in the arguments for the $S$ terms (e.g., writing $S(X,Y)$ instead of $S(XY)$) for clarity.
%When some of the arguments are equal, it takes the simpler form:
%\begin{align}
%S(XX) \am S(YY) = &\ \hv \left( S((\ad X)^2 Y \ot Y)  + S((\ad Y)^2 X \ot X) \right) \notag \\
%&+ \hv S([X,Y], [X,Y]) + K(X,Y)S(X,Y). \label{prod.quad}
%\end{align}

\begin{lem} \label{am.def}
The formula \eqref{prod.full} extends to a symmetric bilinear map $\am \!: A(\g) \times A(\g) \to A(\g)$.
\end{lem}

\begin{proof}
Since both sides of \eqref{prod.full} are linear in each of $A$, $B$, $C$, $D$ and symmetric under swapping $A$, $B$ and $C$, $D$, it remains only to check that $\am$ is well defined, i.e., that the expression given for $S(AB) \am S(v)$ is zero for all $v \in \ker S$.  It is sufficient to check this over an algebraic closure of $k$, where we are reduced to the following computation.

Let $Y, X_1, \ldots, X_r \in \g$ be such that $S(\sum X_i^2) = 0$.  The expression for $S(Y^2) \am \sum S(X_i^2)$ is 
\begin{multline}  \label{wd.sum}
\hv \sum S(((\ad Y)^2 X_i) X_i) + \hv \sum S(((\ad X_i)^2 Y) Y)  \\
+ \hv \sum S([Y,X_i][Y,X_i]) + \sum K(Y, X_i) S(X_iY).
 \end{multline}
As $\sum S(X_i^2) = 0$, $\sum P(X_i^2) = -\hv\sum (\ad X_i)^2$, so the second and fourth terms in \eqref{wd.sum} cancel.
 
 Furthermore, as $S$ is $\g$-equivariant, we have 
 \begin{equation} \label{A.eq}
 [\ad Z, S(AB)] = S([Z,A]B) + S(A[Z,B]) \quad \text{for $A, B, Z \in \g$.}
 \end{equation}
 Adding the first and third term in \eqref{wd.sum}, dividing by $\hv$, and applying this identity twice gives
\[
\left[ \ad Y, \sum S([Y,X_i]X_i)\right] = \frac12 [\ad Y, [\ad Y, \sum S(X_i^2)]] = 0.
\]
In summary, \eqref{wd.sum} is zero.   Therefore, if we write 
$a, a' \in A(\g)$ as $a = S(w)$ and $a' = S(w')$ for $w, w' \in \Sym^2\g$, the value of $a \am a'$ given by \eqref{prod.full} does not depend on the choice of $w, w'$.
\end{proof}

With Lemma \ref{am.def} in hand, we view $A(\g)$ as a commutative $k$-algebra with the product $\am$ defined by \eqref{prod.full}.  

\begin{lem} \label{id.lem}
The identity transformation $e$ of $\g$ is the multiplicative identity in $A(\g)$, i.e., $e \am a = a$ for all $a \in A(\g)$.
\end{lem}

\begin{proof}
First note that $e$ is in $A(\g)$ by Example \ref{S.es}.
We may enlarge our base field and so assume that $k$ is algebraically closed and in particular that $\g$ has an orthonormal basis $\{ X_i \}$.  Combining \eqref{prod.full} and \eqref{A.eq}, we obtain
\begin{multline*}
S(X_i^2) \am S(Y^2) = \frac{\hv}2 [ \ad Y, [\ad Y, S(X_i^2)]] + \hv S((\ad X_i)^2 Y, Y) \\+ K(X_i, Y) S(X_i Y).
\end{multline*}
If we sum both sides over $i$, we have $(\hv + 1) e \am S(Y^2)$ on the left by Example \ref{es.eg} and  $0+\hv S(Y^2) + S(Y^2)$ on the right.  Consequently $S(Y^2) \am e = S(Y^2)$, as required.
\end{proof}

\begin{rmk} \label{DMVC.rmk}
The paper \cite{DeMedtsVC} constructs an algebra $A$ similar to $A(\g)$ that is also a subspace of $\Herm(\g)$, but with a different product, which we denote by $\Mm$ for the moment.  It defines $a \Mm a'  := \Proj_A(a \jord a')$, which differs from our product defined in \eqref{prod.full}.  The analog of \eqref{prod.full} for their multiplication $\Mm$ has additional terms.  For the case where $G$ has type $E_8$, both algebras can be viewed as different ways of adding a unit to the irreducible 3875-dimensional representation.  Since that representation supports a unique $E_8$-invariant product, the difference between our multiplications is necessarily minor.  That is, if our $A(\g)$ is written as $\Unit(V, f)$ in the notation of section \ref{unitizing.sec}, then theirs is $\Unit(V, cf)$ for some invertible $c \ne 1$ in $k$.
\end{rmk}

\subsection*{A Jordan subalgebra}
Suppose that $\lsub$ is an abelian subalgebra of $\g$.  (For example, one could take $\lsub = \h$.)   Define a $k$-linear map
\begin{equation} \label{jord.incl}
i \!: P(\Sym^2 \lsub) \to A(\g) \quad \text{via} \quad i(P(xy)) := S(xy).
\end{equation}
Writing out \eqref{prod.full}, we find that
\[
i(P(xy) \jord P(zw)) = S(xy) \am S(zw),
\]
i.e., $i$ is an algebra homomorphism, and the image of $P(\Sym^2 \lsub)$ is a Jordan subalgebra of $A(\g)$.  (Note that the identity element of $P(\Sym^2 \lsub)$ need not map to the identity element of $A(\g)$, see the proof of Prop.~\ref{nonpa}.)

\begin{lem} \label{jord.inj}
If $\lsub$ is an abelian subalgebra of $\g$ and the Killing form $K$ restricts to be nondegenerate on $\lsub$, then the homomorphism \eqref{jord.incl} is injective.
\end{lem}

Note that when $K\vert_\lsub$ is nondegenerate, the isomorphism $\g \ot \g \iso \g \ot \g^*$ restricts to an isomorphism $\ell \ot \ell \iso \ell \ot \ell^*$ which identifies $P(\Sym^2 \lsub)$ with the Jordan algebra $\Herm(\lsub)$ of symmetric elements in $\End(\lsub)$.

\begin{proof}
The definition of $S$ shows that $i(P(\Sym^2 \lsub))$, as a subspace of $\End(\g)$, acts on $\lsub$ via $i(P(X^2))(Y) = P(X^2)(Y)$ for all $X, Y \in \lsub$.  The nondegeneracy of $K$ then identifies $i(P(\Sym^2 \lsub))$ with the symmetric elements in $\End(\lsub)$.
\end{proof}

\begin{eg} \label{cont}
Suppose $G$ is split and not of type $A_1$ nor $A_2$.   Fix a Chevalley basis for $G$ as in \S\ref{Ag.sec}.  For $H \in \h$ such that $K(H,H)$ is not zero,  the element $u_H := i(P(H^2))/K(H,H)$ is an idempotent in $A(\g)$.   This provides an idempotent in $A(\g)$ for every element of $\mathbb{P}(\h)$ in the complement of the quadric hypersurface defined by $K(X,X) = 0$.
Clearly, if $k$ is infinite, there are infinitely many idempotents in $A(\g)$.

Now, there is a positive root $\gamma$ that is orthogonal to the highest root $\hst$.  For 
the element $S(X_\hst X_\gamma)$, which is nonzero by Lemma \ref{high.wt}\eqref{high.0}, we have 
\[
u_H \am S(X_\hst X_\gamma) = \la_H S(X_\hst X_\gamma) \quad \text{for} \quad
\la_H = \frac{\hv ((\hst + \gamma)(H))^2}{2 K(H,H)}.
\]
The map $H \mapsto \la_H$ is a rational function $\h \dashrightarrow k$ that is not constant and therefore is dominant.  In particular, the collection of eigenvalues of the maps $x \mapsto u \am x$ as $u$ varies over the idempotents of $A(\g)$ is not contained in $\{ 0, \frac12, 1 \}$, and therefore $A(\g)$ is not power-associative, cf.~\cite[Ch.~V]{Schfr}.
\end{eg}

%%%%%%%%%%%%%%%%%%%%%%%%%%%%%%%%%%%%%%%%%%%%%%%%%%%%%%%%%%%%%%%%%
\section{$A(\g)$ as an algebra obtained by adding a unit}  \label{pa.sec}

The usual trace form $\Tr \colon \End(\g) \to k$ is linear and $G$-invariant.  We use it to define a counit, in the sense of the appendix, as $\counit := \textstyle\frac1{\dim G} \Tr$ so that $\counit(e) = 1$, for $e = \Id_\g$ the identity element in $A(\g)$ (Lemma \ref{id.lem}).  Thus we obtain a bilinear form $\tbil$ on $A(\g)$ via \eqref{tbil.def}, $\tbil(a, a') := \counit(a \am a')$.  The form $\tbil$ is evidently $G$-invariant (because $\Tr$ and $\am$ are), symmetric (because $\am$ is commutative), and bilinear.

\begin{eg} \label{eq5}
For $X, Y \in \g$, Example \ref{tr.Sx} gives 
\begin{equation} \label{tbil.Sx}
\tbil(e, S(XY)) = \frac{\hv + 1}{\dim G} K(X,Y) \quad \text{for $X, Y \in \g$.}
\end{equation}
We also note for future reference:
\begin{align*}
\tbil(S(X^2), S(Y^2)) &= \textstyle\left( \frac{\hv + 1}{\dim G} \right) \left( -\hv K([X,Y], [X,Y]) + K(X,Y)^2 \right) \\
&= \textstyle\left(\frac{\hv + 1}{\dim G} \right) K(S(X^2) Y, Y).
\end{align*}
\end{eg}

Using the counit $\counit$ defined above, the algebra $A(\g)$ can be viewed as an algebra $\Unit(V, f)$ as in the appendix, where $V$ is the vector space $\ker \counit$ endowed with the commutative product $\cdot$ and $f$ as defined in \eqref{AV}.  With this notation, we prove:

\begin{prop} \label{nonpa}
If $G$  is not of type $A_1$ nor $A_2$, then:
\begin{enumerate}
\item \label{nonpa.nz} The multiplication $\cdot$ on $V$ is not zero.
\item \label{nonpa.nonpa} Neither $V$\! nor $\Unit(V, cf)$ is power-associative for any $c \in k$.
\end{enumerate}
\end{prop}

For the excluded cases of $A_1$ and $A_2$, see Examples \ref{A1} and \ref{A2} respectively.
 
\begin{proof}
For each claim, we may enlarge $k$ and so assume that the Lie algebra $\h$ of some maximal torus in $G$ has an orthonormal basis $X_1, \ldots, X_\ell$.  We set $B := i(P(\Sym^2 \h))$.

We begin with \eqref{nonpa.nz}.  By \eqref{tbil.Sx}, for $i \ne j$, $S(X_i X_j)$ is in $V$.  On the other hand, if $\ell \ge 3$,
\[
S(X_1 X_2) \am S(X_1 X_3) = i(P(X_1X_2) \jord P(X_1 X_3)) = \textstyle\frac14 S(X_2 X_3) \ne 0
\]
and we are done.  If $\ell = 2$, then $e' := S(X_1^2 + X_2^2)$ is the identity element in $B$ by Example \ref{jord.eg}, yet 
\[
s := \tbil(e, e') = 2\frac{\hv + 1}{\dim G} = \frac{\hv + 1}{h + 1}
\]
is not 1 because $G$ is not of type $A_2$.  Then $e' - se$ is in $V$ and $(e' - se) \cdot S(X_1X_2) = (1 - s) S(X_1 X_2) \ne 0$, verifying \eqref{nonpa.nz}.

For \eqref{nonpa.nonpa}, put $r := (\hv + 1)/(\dim G)$, a rational number whose denominator is not divisible by $\car k$.  Since $\hv \le h$, $0 < r \le 1/2$.  Define a map $S^+ \!: \Sym^2 \g \to V$ by $S^+(p) = S(p) - \counit(S(p))\,e$.  Applying Example \ref{eq5}, we find:
\begin{equation} \label{tau.splus}
\tbil(S^+(X^2), S^+(X^2)) = r (1 - r) K(X,X)^2 \quad \text{for $X \in \g$.}
\end{equation}
Therefore $\tbil$ (equivalently, $f$) is not zero on $V$, and in particular $f$ is not alternating.

Set $b := i(P(X_1^2) + t P(X_2^2))$ where $t \in k$ is neither 0 nor 1, so $(1,0)$, $b$, and $b^2 = i(P(X_1^2) + t^2 P(X_2^2))$ are linearly independent (Lemma \ref{jord.inj}).  
Let $B$ be the subalgebra of $\Unit(V,f)$ generated by $(1, 0)$ and $b$.  Then $B = \Unit(V \cap B, f \vert_{V \cap B})$, and $B$ is power-associative because $b$ generates a Jordan subalgebra of $\Unit(V,f)$. 

We have already observed in Example \ref{cont} that $\Unit(V, f)$ is not power-associative, so we fix $c \ne 1$.  By Proposition \ref{unit.unique}, $\Unit(V \cap B, cf\vert_{V \cap B})$ is not strictly power-associative, and so $\Unit(V, cf)$ is not strictly power-associative either.  It follows that $\Unit(V, cf)$ is not power-associative, because $\car k \ne 2, 3, 5$ and $\Unit(V, cf)$ is commutative.  The case $c = 0$ gives that $V$ itself is not power-associative.
\end{proof}

As opposed to defining the product on $A(\g)$ via \eqref{prod.full}, one could build $A(\g)$ ``from below'' by starting with a $G$-invariant commutative product $\cdot$ on a representation $V$ and a $G$-invariant bilinear form $f$ and setting $A(\g)$ to be $\Unit(V, f)$.  In case $G$ has type $E_8$ and $V$ is the irreducible 3875-dimensional representation, both $\cdot$ and $f$ are uniquely determined up to a factor in $k^\times$.  But only the scalar factor on $f$ matters (Remark \ref{unit.nomu}), and \eqref{nonpa.nonpa} says that the resulting algebra is not power-associative, no matter what choice one makes for that parameter.

Similarly, the conclusion of Lemma \ref{assoc} below would be unchanged by multiplying $f$ by a scalar factor, as is clear from Proposition \ref{app.assoc}.

%%%%%%%%%%%%%%%%%%%%%%%%%%%%%%%%%%%%%%%%%%%%%%%%%%%%%%%%%
\section{Associativity of the bilinear form $\tbil$} \label{assoc.sec}

The following property of the symmetric bilinear form $\tbil$ on $A(\g)$ is sometimes described as saying that ``$\tbil$ is associative'', especially in the context of Dieudonn\'e's Lemma as in \cite[pp.~199, 239]{Jac:J}.

\begin{lem} \label{assoc}
The bilinear form $\tbil$ on $A(\g)$ satisfies
\begin{equation} \label{t.assoc}
\tbil(a \am a', a'') = \tbil(a,a' \am a'') \quad \text{for all $a, a', a'' \in A(\g)$.}
\end{equation}
\end{lem}

\begin{proof}
It suffices to verify this in the case $a = S(X^2)$, $a' = S(Y^2)$, and $a'' = S(Z^2)$ for $X, Y, Z \in \g$.  Expanding out following the definitions, one finds:
\begin{multline}  \label{assoc.1}
\left( \textstyle\frac{\dim G}{\hv+1} \right) \tbil(a \am a', a'') = \\(\hv)^2 \left( K((\ad Z)^2 Y, (\ad X)^2 Y) + K((\ad Z)^2 X, (\ad Y)^2 X) + K((\ad Z)^2 [X,Y], [X,Y]) \right) \\
+ \hv \left( K(Z,Y) K(Z, (\ad X)^2 Y) + K(X,Y) K((\ad Z)^2 Y, X) + K(Z,X) K(Z, (\ad Y)^2 X) \right) \\
+ \hv K([X,Y],Z)^2 + K(X,Y) K(X,Z) K(Y,Z).
\end{multline}
Put $\psi$ for the alternating trilinear form $\psi(A,B,C) = K([AB],C)$ on $\g$ and observe that $\Psi := \psi([X,Z], [X,Y], [Y,Z])$ is invariant under permutations of the variables $X$, $Y$, $Z$.  We have:
\begin{align*}
\Psi &= K([X,Z], [[[X,Y],Y],Z]) + K([X,Z],[Y,[[X,Y],Z]]) \\
&= -K((\ad Z)^2 Y, (\ad Y)^2 X) - K([[X,Z],Y], [Z,[X,Y]]) 
\end{align*}
Adding this equation to the same equation with $X$ and $Y$ swapped gives that $-2\Psi$ is the first term in parentheses on the right side of \eqref{assoc.1}.  That is, 
\begin{multline} \label{assoc.2}
\left(\textstyle\frac{\dim G}{\hv+1} \right)\tbil(a \am a', a'') = -2(\hv)^2 \Psi -\hv E \\ +\hv \psi(X,Y,Z)^2  
 + K(X,Y) K(X,Z) K(Y,Z)
\end{multline}
with 
\[
E = \psi(X,Y,[X,Z]) K(Y,Z) + \psi(X,Z,[Y,Z])K(X,Y) + \psi(Y,X,[Y,Z])K(X,Z).
\]
Each of the four terms on the right side of \eqref{assoc.2} is unchanged when we swap $X$ and $Z$, and therefore the claim is verified.
\end{proof}

\begin{rmk}
Here is another argument to show associativity of $\tbil$ that works when $G$ has type $E_8$.  In that case, $A(\g) = ke \oplus V$ where $V$ is an irreducible representation of $G$ (Lemma \ref{decomp.lem}), the restriction $f$ of $\tbil$ to $V$ is nondegenerate, and the space $(V^* \otimes V^* \otimes V^*)^G$ of $G$-invariant trilinear forms on $V$ is 1-dimensional.  It follows then that the linear maps defined by sending $v \ot v' \ot v'' \in \otimes^3 V$ to $f(v \cdot v', v'')$ and $f(v, v' \cdot v'')$ agree up to a scalar factor, where $f$ is the restriction of $\tbil$ to $V$.  The two cubic forms are nonzero (Prop.~\ref{nonpa}\eqref{nonpa.nz}) and agree when $v = v' = v''$ is a generic element of $V$, so the two forms agree in general, i.e., $f$ is associative with respect to the product $\cdot$, whence $\tbil$ is associative with respect to the product $\am$ on $A(\g)$ by Prop.~\ref{app.assoc}.
\end{rmk}
%%%%%%%%%%%%%%%%%%%%%%%%%%%%%%%%%%%%%%%%%%%%%%
\section{$A(\g)$ as a representation of $G$} \label{rep.sec}

The counit $\counit$ gives a direct sum decomposition $A(\g) = ke \oplus V$ as a representation of $G$.  In this section, we describe $V$ as a representation of $G$ and show that its dimension and character depend only on the root system of $G$ and not on the field $k$ nor even the characteristic of $k$.  We use the notion of Weyl module recalled in \S\ref{gen.sec}.  

\begin{eg}[$A(\sl_2)$] \label{A1}
Suppose $G$ is split of type $A_1$, so $\g = \sl_2$.  By hypothesis, $\car k$ is zero or at least 5, so the Weyl module $V(4)$ of $G$ with highest weight $4$ is irreducible over $k$ \cite{WinterSL2}.  It is a submodule of $\Herm(\g)$ generated by $P(X_\hst^2)$ and $\Herm(\g)/k$ is $V(4)$ by dimension count.  As $A(\g)$ does not meet $V(4)$ (Cor.~\ref{high.2hst}), it follows that $A(\sl_2) = k$ as a vector space, spanned by $\Id_\g$, i.e., $A(\sl_2)$ is identified with $k$ as a $k$-algebra.  
\end{eg}

The notion of Weyl module still makes sense when $G$ is not assumed to be split. In that case, one still picks a maximal torus $T$ defined over $k$.  Pick any Borel subgroup $B$ of $G \times \kalg$ containing $T$, equivalently, pick a cone of dominant weights in the character lattice $T^*$.  There is a natural action of the Galois group $\Aut(\kalg/k)$ on $T^*$, which maps the cone to itself if and only if $B$ is defined over $k$.  In any case, there is a canonical way to modify the action using the Weyl group to produce a new action of $\Aut(\kalg/k)$ on $T^*$ that does leave the cone invariant, see \cite[6.2]{BoTi} or \cite[\S3.1]{Ti:R}.  (This action permutes the simple roots and is determined by how it does so, and therefore is equivalent to an action of $\Aut(\kalg/k)$ on the Dynkin diagram of $G$.)

\begin{table}[hbt]
\[
\begin{array}{c|cccccc}
\text{type of $G$}&A_2&G_2&F_4&E_6&E_7&E_8 \\  \hline
\text{Dual Coxeter number $\hv$}&3&4&9&12&18&30 \\
\text{Coxeter number $h$}&3&6&12&12&18&30 \\
\text{Dominant weight $\la$}&\omega_1+\omega_2&2\omega_1&2\omega_4&\omega_1+\omega_6&\omega_6&\omega_1 \\
\text{Dim.~of irred.~rep.~$ L(\la)$}&8&27&324&650&1539&3875
\end{array}
\]
\caption{Data for some exceptional groups $G$.  The fundamental dominant weights in the formula for $\la$ are numbered as in \cite{Bou:g4}.} \label{dual.coxeter}
\end{table}

Suppose that $\la \in T^*$ is a dominant weight, is in the root lattice, and is fixed by the action of $\Aut(\kalg/k)$ on the dominant weights.  (This holds, for example, for $\la = 2\hst$ and any $G$, or for $G$ and $\la$ as in Table \ref{dual.coxeter}.)  Then there is a unique representation of $G$ over $k$ that becomes isomorphic to $V(\la)$ (respectively, $H^0(\la)$; resp.~$L(\la)$) over $\kalg$.  This is proved in \cite[Th.~3.3]{Ti:R} for the irreducible $L(\la)$, and the same argument works for the other two representations.  Therefore, for such a $\la$, it makes sense to use the same notation also for the representation of $G$ over $k$

\begin{prop} \label{decomp.lem}
Suppose
$G$ is as in Table \ref{dual.coxeter}, $\car k = 0$, or 
\[
\car k \ge \textstyle\binom{\dim G + 1}{2}/(\rank G).
\]
As a representation of $G$, $A(\g)$ is a direct sum of pairwise non-isomorphic irreducible modules and $\Herm(\g) = A(\g) \oplus V(2\hst)$.    Furthermore, 
if $G$ and $\la$ are as in Table \ref{dual.coxeter}, then $A(\g) = k \oplus L(\la)$.
\end{prop}

Note that the displayed lower bound on $\car k$ grows like $(\rank G)^3$, so it is somewhat more restrictive than our global hypothesis that $\car k = 0$ or at least $h + 2$, because $h + 2$ grows like $\rank G$.

\begin{proof}[Proof of Proposition \ref{decomp.lem}]
We first address the case where $k$ is algebraically closed of characteristic zero.  Then $\Herm(\g) \cong k \oplus J \oplus L(2\hst)$ where $k$ is the span of $e$ and $L(2\hst)$ is the $G$-submodule generated by $P(X_\hst^2)$, which does not belong to $A(\g)$ by Corollary \ref{high.2hst}.  Writing $J$ as a sum of irreducible representations $\oplus_i L(\la_i)$, the values of $\la_i$ are known.  If $G$ is from Table \ref{dual.coxeter}, then $J = L(\la)$ is described in \cite{CohendeMan}, where it is denoted by $Y_2^*$.  If $G$ has type $A_1$, then $J = 0$.  Otherwise, $J$ is a sum of three irreducible components for type $D_4$ or two for the other types, see \cite{Vogel} and \cite{LandsMan:univ} for more on this decomposition and related subjects.  In all cases, the $\la_i$ are distinct, are not zero, and are maximal weights for $J$.

To complete the proof for this $k$, we must verify that $J \subseteq A(\g)$.  The bulk of the $\la_i$'s are of the form $\hst + \beta$ for a root $\beta$ obtained by the following procedure.  Take the Dynkin diagram for $G$, delete all simple roots that are not orthogonal to the highest root $\hst$, and select one of the connected components that remains.  It corresponds to a subsystem of the root system of $G$ and is the subsystem for a subalgebra $\g'$ of $\g$ normalized by our chosen maximal torus $T$.  (One says that $\g'$ is a \emph{regular} subalgebra.)  Put $\beta$ for the highest root of $\g'$ (in the ordering induced from the chosen ordering on the weights of $G$).  The element $S(X_\hst X_\beta)$ is not zero by Lemma \ref{high.wt}\eqref{high.0}, so we conclude that $S(X_\hst X_\beta)$ is a highest weight vector and $L(\hst + \beta) \subseteq A(\g)$.

\newcommand{\pshort}{\Phi_S}
For types $A_n$ and $C_n$ with $n \ge 2$, one component of $J$ is of the form considered in the previous paragraph (and so we have shown that it belongs to $A(\g)$) and the other is $\la$ for $\la$ the highest short root.  For type $A_n$, we set 
\begin{gather*}
\beta_j := \alpha_1 + \alpha_2 + \cdots + \alpha_j \quad \text{for $1 \le j < n$,} \\
\gamma_j := \alpha_j + \alpha_{j+1} + \cdots + \alpha_n \quad \text{for $1 < j \le n$, and} \\
p := 2 \hv X_\hst H'_{\omega_1-\omega_n} - \sum_{j=1}^{n-1} [X_\hst, X_{-\beta_j}] X_{\beta_j} + \sum_{j=2}^n [X_\hst, X_{-\gamma_j}] X_{\gamma_j} \quad \in \Sym^2 \g,
\end{gather*}
where the $\alpha_i$ are the simple roots as numbered in \cite{Bou:g4} and $\omega_i$ is the corresponding fundamental dominant weight.
%\[
%p := X_\hst H_{\omega_1 - \omega_n} - \sum_{j=1}^{n-1} [X_\hst, X_{-\beta_j}] X_{\beta_j} + \sum_{j=2}^n [X_\hst, X_{-\gamma_j}] X_{\gamma_j} \quad \in \Sym^2 \g.
%\]
For type $C_n$, $\la = \hst - \beta$ for $\beta$ the simple root not orthogonal to $\hst$, and we set
\[
p := 2X_\hst X_{-\beta} -  \sum_{\mu \in \pshort} [X_{-\mu}, X_\hst][X_{\mu}, X_{-\beta}] \quad \in \Sym^2 \g
\]
for $\pshort$ the set of short roots.
In either case, $p$ has weight $\la$ and a lengthy verification shows that $S(p)$ is not zero and is fixed by each unipotent subgroup of $G$ corresponding to a positive root, verifying that $\la$ is a highest weight vector in $A(\g)$ and therefore that $L(\la)$ is a summand of $A(\g)$ and completing the proof for this $k$.

\medskip
Next suppose that $k$ is algebraically closed of characteristic $p \ne 0$.  We transfer the results proved over $\C$ to $k$ via $R := \Z_{(p)}$.  We use subscripts $\C$, $k$, $R$ to denote corresponding objects over these three rings.  For example, let $G_R$ denote the unique split reductive group scheme over $R$ with the same root datum as $G$, so $G_R \times k \cong G$, and put $\g_R := \Lie(G_R)$.  For each dominant weight $\eta$, there is a Weyl module $V_R(\eta)$ of $G_R$ defined over $R$ such that $L_\C(\eta) = V_R(\eta) \times \C$ and $V_k(\eta) = V_R(\eta) \times k$.

The representations $V_k(\la_i)$ and $V_k(2\hst)$ are irreducible.  If $G$ is from Table \ref{dual.coxeter}, then this fact is contained in the tables in \cite{luebeck}.  Otherwise, $p \ge (\dim \Herm(g))/(\rank G)$, and every representation of $G$ of dimension at most $\dim \Herm(\g)$ is semisimple \cite[Cor.~1.1.1]{McNinch:ss}.  A semisimple Weyl module is irreducible \cite[Cor.~II.2.3]{Jantzen}, proving the claim.  The same argument shows that $V_k(\la_i)$ is irreducible for all $i$.

The map $S \!: \Sym^2 \g \to \Herm(\g)$ is defined over $R$ and the dimension of its image over $\C$ is at least as large as its image over $k$ by upper semicontinuity of dimension.  As the arguments above show that the irreducible representation $L_k(\la_i)$ belongs to $A(\g)$ over $k$ for all $i$ and there are no nontrivial extensions among the $L_k(\la_i)$ \cite[II.4.13]{Jantzen}, we conclude that $A(\g) \cong k \oplus \left( \oplus_i L_k(\la_i) \right)$ as a representation of $G$.

As a quotient of vector spaces, $\Herm(\g) / A(\g)$ is a representation of $G$ with highest weight $2\hst$, so there is a nonzero homomorphism $V_k(2\hst) \to \Herm(\g)/A(\g)$.  The preceding arguments showed that the dimension of $A(\g)$, hence the dimensions of both the domain and codomain of the map, do not depend on $k$.  Since $V_k(2\hst)$ is irreducible, the map is injective, so an isomorphism by dimension count.  There are no non-trivial extensions among the irreducible representations appearing in the composition series for $\Herm(\g)$, whence the claim in the second sentence of the proposition.

\medskip

Finally, drop the hypothesis that $k$ is algebraically closed; in particular $G$ need not be split.  The center of $G$ acts trivially on $\Herm(\g)$, so we may assume that $G$ is adjoint.  We view $G$ and the representation $A(\g)$ as being obtained from a representation $A(\g_0)$ of the unique split form $G_0$ of $G$ over $k$ by twisting by a 1-cocycle $\eta$ in Galois cohomology $Z^1(\Aut(\kalg/k), \Aut(G_0))$ as in \cite[\S{III.1.3}]{SeCG}.  (Recall that the component group of $\Aut(G_0)$ can be identified with the automorphism group of the Dynkin diagram as in \cite[Ch.~XXIV, 1.3, 3.6, 5.6]{SGA3.3:new} or \cite[\S16.3]{Sp:LAG}, and the image of $\eta$ in $Z^1(\Aut(\kalg/k), \Aut(G_0)/\Aut(G_0)^\circ)$ encodes the $*$-action.)  If $G$ is not of type $D_4$, then the $\la_i$'s are each fixed by the $*$-action and belong to the root lattice, hence each representation $L(\la_i)$ of $G_0$ is naturally compatible with the twisting by $\eta$, giving an irreducible representation of $G$ defined over $k$, as discussed before the statement of the proposition.  For $G$ of type $D_4$, $A(\g_0) = k \oplus L(\la_1) \oplus L(\la_2) \oplus L(\la_3)$ as a representation of $G_0$, and the $*$-action permutes the $\la_i$'s according to its action on the three terminal vertices in the Dynkin diagram.  As in \cite[Th.~7.2]{Ti:R}, we find that the representation $A(\g)/k$ of $G$, which is obtained by twisting the representation $A(\g_0)/k$ of $G_0$ by $\eta$, is a sum of $o$ distinct irreducible representations of $G$ over $k$, where $o$ is the number of orbits of $\Aut(\kalg/k)$ on the set $\{ \la_1, \la_2, \la_3 \}$.
\end{proof}

For $G$ as in Table \ref{dual.coxeter}, $\dim_k A(\g) = 1 + \dim L(\la)$ as provided in the table.  For $G$ of type $A$, $B$, $C$, or $D$ and under the hypotheses of Proposition \ref{decomp.lem}, we have
\begin{equation}
\dim_k A(\g) = \textstyle\binom{\dim G + 1}{2} - \dim V(2\hst),
\end{equation}
where $\dim V(2\hst)$ is given by the Weyl dimension formula.  

 %%%%%%%%%%%%%%%%%%%%%%%%%%%%%%%%%%%%%%%%%%%%%%%%%%%%%%%%%%%
 \section{$\tbil$ is nondegenerate and $A(\g)$ is simple} \label{ndeg.sec}
 
 \begin{prop} \label{metrized}
$\tbil$ is nondegenerate on $A(\g)$.
 \end{prop}
 
 Some authors would summarize Lemma \ref{assoc} and Proposition \ref{metrized}, which say that $A(\g)$ has an associative and nondegenerate symmetric bilinear form, by saying ``$A(\g)$ is metrized''.

The proof leverages the following.

\begin{eg} \label{high.pd}
Suppose we are in the situation of Lemma \ref{high.wt}\eqref{high.0}, i.e., $G$ is split and $\alpha$, $\beta$ are orthogonal roots and $\alpha+ \beta$ is not a root.  Recall from \eqref{K.spst} that $K(X_{\gamma}, X_{-\gamma})$ is not zero in $k$ for every root $\gamma$ (and is positive when $k \subseteq \R$) and that $S(X_\alpha X_\beta) X_{-\alpha} = \nu_\alpha \hv X_\beta$.  Bilinearizing Example \ref{eq5}, we have
\begin{align*}
\tbil(S(X_\alpha X_\beta), S(X_{-\alpha} X_{-\beta})) &= \textstyle\left( \frac{\hv+1}{\dim G} \right) K(S(X_\alpha X_\beta) X_{-\alpha}, X_{-\beta}) \\
& = \textstyle\left( \frac{\hv+1}{\dim G} \right) 2 (\hv)^2 \nu_\alpha \nu_\beta,
\end{align*}
where the second equality is by \eqref{K.spst}.
Note that this is not zero in $k$.
Moreover, in case $k \subseteq \R$, the expression is positive.
\end{eg}

\begin{proof}[Proof of Proposition \ref{metrized}]
We may enlarge $k$ and so assume that $G$ is split.
Recall from \S\ref{rep.sec} that $A(\g) = ke \oplus (\oplus_i L(\la_i))$ for a set of dominant weights $\{ \la_i \}$.  This sum is an orthogonal sum with respect to $\tbil$, and therefore it suffices to verify the claims for the restriction of $\tbil$ to each $L(\la_i)$.

Pick $p \in \Sym^2 \g$ such that $S(p)$ is a highest weight vector in $L(\la_i)$.  In case $\la_i = \hst + \beta$ for some positive root $\beta$ orthogonal to $\hst$, we take $p := S(X_\hst X_\beta)$.  Otherwise, $\la$ is the highest short root and $G$ has type $A_n$ for $n \ge 3$ or $C_n$ for $n \ge 2$; in that case we take $p$ to be as in the proof of Proposition \ref{decomp.lem}.

Define $\theta$ to be the automorphism of $\g$ such that $\theta\vert_\h = -1$ and $\theta(X_\gamma) = X_{-\gamma}$ for each root $\gamma$.  Then $S(\theta p)$ is a lowest weight vector in $L(\la_i)$.  We verify that $\tbil(S(p), S(\theta p))$ is not zero; in the first case this is Example \ref{high.pd}, and in the second case a calculation is required.  Therefore, the restriction of $\tbil$ to $L(\la_i)$ is not zero, so it is nondegenerate, verifying the claim.
\end{proof}

\begin{cor} \label{metrized.pd}
If $k = \R$ and $G$ is compact, then $\tbil$ is positive-definite on $A(\g)$.
\end{cor}

\begin{proof}
We continue the notation of the proof of Proposition \ref{metrized}.  We view $\g$ as the subalgebra of the split complex Lie algebra consisting of elements fixed by the Cartan involution obtained by composing $\theta$ with complex conjugation as in \cite[\S{IX.3.2}]{Bou:g7}.  Then $v := p + \theta p$ is in $\Sym^2 \g$, $S(v)$ is in $L(\la_i)$, and $\tbil(S(v), S(v)) = 2 \tbil(S(p), S(\theta p)) > 0$.  As $G$ is compact, every nonzero $G$-invariant bilinear form on $L(\la_i)$ is definite, so $\tbil$ is positive definite on $L(\la_i)$.
\end{proof}

\begin{proof}[Alternative proofs for exceptional groups]
Here are very short proofs of Proposition \ref{metrized} and Corollary \ref{metrized.pd} in case $G$ belongs to Table \ref{dual.coxeter}.  By \eqref{tau.splus}, $\tbil$ is not zero on the irreducible representation $V$, so it is nondegenerate on $V$, hence on all of $A(\g)$.
Suppose $G$ is a compact real form, so every nonzero $G$-invariant bilinear form on the irreducible representation $V$ is definite, as can be seen by averaging.  In particular $\tbil$ is definite on $V$, so positive definite on $V$ by \eqref{tau.splus}.  Corollary \ref{metrized.pd} follows.
\end{proof}

One says that $G$ is \emph{isotropic} if it contains a copy of the 1-dimensional split torus $\Gm$ defined over $k$, and \emph{anisotropic} otherwise.  In case $k = \R$,  $G$ is anisotropic if and only if it is  compact.  The following example provides  something like a converse to Corollary \ref{metrized.pd}.

\begin{eg} \label{metrized.iso}
Suppose $G$ is not of type $A_1$ and $G$ is isotropic; we claim that $\tau$ is isotropic. As $G$ is not of type $A_1$, $A(\g)/k$ is not the trivial representation of $G$ as in the proof of Prop.~\ref{decomp.lem}, so $G$ acts on it with finite kernel.  It follows that there is a nonzero subspace $U$ of $A(\g)$ on which $\Gm$ acts with only positive weights or only negative weights, implying that $\tau(u, u') = 0$ for $u, u' \in U$, i.e., $\tau$ is isotropic.
\end{eg}

The next example shows that the case $k = \R$ in Corollary \ref{metrized.pd} is somewhat special.

\begin{eg}
We will show that $\tau$ may be isotropic, even if the group $G$ is anisotropic.   Specifically, let $k$ be a number field and pick an odd number $n \ge 3$.  There is an associative division algebra $D$ with center $k$ such that $\dim_k D = n^2$.  The group $G = \SL_1(D)$ of norm 1 elements of $D$ is simply connected of type $A_{n-1}$ and is anisotropic.  However, the group is split at every real place, so $\tau$ is isotropic at every real place (Example \ref{metrized.iso}).  As $\dim A(\g) \ge 1 + \dim \g > 5$, the form $\tau$ is isotropic over $k$ by the Hasse-Minkowski Theorem.
\end{eg}

We conclude the section with another corollary of Proposition \ref{metrized}.

\begin{cor}  \label{metrized.simple} $A(\g)$ is a simple $k$-algebra.
\end{cor}

\begin{proof}
The nondegeneracy of $\tbil$ and Proposition \ref{decomp.lem} verify the hypotheses of Proposition \ref{unit.simple}.
\end{proof}

 %%%%%%%%%%%%%%%%%%%%%%%%%%%%%%%%%%%%%%%%%%%%%%%%%%%%%%%%%%%
 \section{The group scheme $\Aut(A(\g))$}
 
There is a natural homomorphism $G \to \Aut(A(\g))$.  It has a finite kernel the center of $G$, and it is injective if and only if $G$ is adjoint.  The point of the following result is that in some cases this homomorphism is an isomoprhism.

\begin{prop} \label{irred.prop}
If $G$ has type $F_4$ or $E_8$, then $\Aut(A(\g)) = G$.
\end{prop}

It follows trivially that for $G$, $G'$ of type $F_4$ or $E_8$, we have: $G \cong G'$ if and only if $A(\g) \cong A(\g')$.

\begin{proof}[Proof of Proposition \ref{irred.prop}]
The number $\dim A(\g)$ is not zero in $k$, so as in Example \ref{counit.canonical} $\Aut(A(\g))$ is the sub-group-scheme of $\GL(V)$ preserving the commutative product $\cdot$ on $V$ (nonzero by Prop.~\ref{nonpa}\eqref{nonpa.nz}) as well as the $G$-invariant bilinear form.  In case $G$ has type $F_4$ or $E_8$, it is known that $G$ is the automorphism group of this product by \cite[Lemma 5.1, Remark 5.5, and \S7]{GG:simple}.
\end{proof}

Here is what happens when the argument in the preceding proof is applied to $G$ of the other types in Table \ref{dual.coxeter}: For $G$ adjoint of type $E_6$, the argument shows that $G$ is the identity component of $\Aut(A(\g))$.  For $G$ of type $G_2$ or $E_7$, there is a copy of $\SO_7$ or $\Sp_{56}/\mu_2$ in $\GL(V)$ containing $G$ and preserving a nontrivial linear map $V \otimes V \to V$; as $G$ preserves a two-dimensional space of such products, the argument provided here is inconclusive in these cases.

For type $A_2$, $\Aut(A(\g))$ is the orthogonal group $O(\g)$, whose identity component has type $D_4$, see Example \ref{A2}.

\begin{rmk} \label{pa.crazy}
Let $B$ be a simple, commutative, and power-associative algebra over $\C$.  Then by \cite{Alb:thpa} and \cite{Kokoris:Annals}, $B$ is a Jordan algebra.  The classification of such from \cite[p.~204, Cor.~2]{Jac:J} or \cite[\S\S13, 14]{Sp:jord} shows that the identity component of $\Aut(B)$ cannot be a simple group of type $G_2$, $E_6$, $E_7$, or $E_8$.  

This provides an alternative argument that $A(\g)$ is not power-associative when $G$ is simple of type $G_2$, $E_6$, $E_7$, or $E_8$ over $\C$, because $A(\g)$ is simple (Cor.~\ref{metrized.simple}) and commutative.
(Compare Example \ref{cont}.)
\end{rmk}

%%%%%%%%%%%%%%%%%%%%%%%%%%%%%%%%%%%%%%%%%%%%%%%%%%%%%%%%%
\section{Construction \#2: $A(\g)$ in $\End(V)$} \label{const2}

In this section, we leverage a common property of exceptional groups $G$ observed by Okubo to describe $A(\g)$ inside of $\End(V)$ for certain small $V$.

Suppose for this paragraph that $k = \C$ and $\pi \!: G \to \GL(V)$ is a representation.  The maps $X \mapsto \Tr(\pi(X)^d)$ are $G$-invariant homogeneous polynomial functions on $\g$.  It is standard that $k[\g]^G$ is a polynomial ring with homogeneous generators.  The smallest nonconstant generator can be taken to be $X \mapsto K(X,X)$ of degree 2, and therefore an identity of the form $\Tr(\pi(X)^2) = c_\pi K(X,X)$ for all $X \in \g$, where $c_\pi$ depends on $\pi$, as in Lemma \ref{qcas}\eqref{qcas.tr} is inevitable.  Similarly, for $G$ as in Table \ref{dual.coxeter}, the homogeneous generators of $k[\g]^G$ are $X \mapsto K(X,X)$ of degree 2, for type $A_2$ one of degree 3, and no generators of degree 4, and therefore there is an identity of the form $\Tr(\pi(X)^4) = \alpha_\pi K(X,X)^2$ for $X \in \g$, where $\alpha_\pi$ depends only on $\pi$.

Okubo calculated the value of $\alpha_\pi$ in \cite{Okubo:quartic} in case $k = \C$, see also \cite{Meyberg:Okubo}.  Here we note that the same result holds over our more general $k$.   In this section, let $R$ denote the local ring $\Z_{(\car k)}$ as in Lemma \ref{iprod.R}.

%Let $\pi \!: G \to \GL(B)$ be a representation.
%Recall that, for $G$ simple, the ring $k[\g]^G$ of $G$-invariant polynomial functions on $\g$ is a polynomial ring with homogeneous generators, where the generators are defined over a localization of $\Z$ mapping to $k$, hence have the same degree over $k$ and $\C$ \cite[p.~297]{Dem:inv}.  The generator of smallest degree is $X \mapsto K(X,X)$ of degree 2, and therefore an identity of the form $\Tr(\pi(X)^2) = c_\pi K(X,X)$ for all $X \in \g$, where $c_\pi$ depends on $\pi$, as in Lemma \ref{qcas}\eqref{qcas.tr} is inevitable.  Similarly, for $G$ as in Table \ref{dual.coxeter}, the homogeneous generators of $k[\g]^G$ are $X \mapsto K(X,X)$ of degree 2, for type $A_2$ one of degree 3, and no generators of degree 4, and therefore there is an identity of the form $\Tr(\pi(X)^4) = \alpha_\pi K(X,X)^2$ for $X \in \g$, where $\alpha_\pi$ depends only on $\pi$.  

\begin{lem} \label{2.4a}
Suppose $G$ is one of the types listed in Table \ref{dual.coxeter} and that the representation $\pi \!: G \to \GL(V)$ is equivalent to $H^0(\la)$ or $V(\la)$ over the algebraic closure of $k$ for some dominant weight $\la$.  Put $\mu_\pi := \bil{\la}{\la+\sumroots}$.  If the rational number
\[
\alpha_\pi := \frac{(6\mu_\pi - 1)\mu_\pi \dim V}{2(2 + \dim G) (\dim G)}
\]
belongs to $R$,
then
$\Tr(\pi(X)^4) =  \alpha_\pi \, K(X,X)^2$
for all $X \in \g$.  If additionally $G$ does not have type $A_2$, then
\begin{multline} \label{2.4a1}
\Tr(\pi(X)^2 \pi(Y)^2) = -\frac{\mu_\pi \dim V}{6 \dim G} K([X,Y], [X,Y]) \\
+ \frac{2 \alpha_\pi}{3} K(X,Y)^2 + \frac{\alpha_\pi}{3} K(X,X) K(Y,Y)
\end{multline}
for $X, Y \in \g$.
\end{lem}

\begin{proof}[Sketch of proof]
Similar to the proof of Lemma \ref{qcas}.  Let $G_R$, $\pi_R$ be lifts of $G$, $\pi$ to $R$ and put $K_R$ for the Killing form on $\g_R$.  The map $X \mapsto \Tr(\pi_R(X)^4) - \alpha_\pi K_R(X,X)$ is a polynomial function on $\g_R$ (an element of $R[\g_R]$) that vanishes over $\C$ by Okubo, so it is 0 in $R[\g_R]$.  Similarly, equation \eqref{2.4a1} holds over $\C$, see \cite[p.~284]{Meyberg:Okubo}, so it too holds over $R$.
\end{proof}

\begin{eg}
For the adjoint representation, we have 
\[
\alpha_{\Ad} = \frac{5}{2(2 + \dim G)}, 
\]
which belongs to $R$ for $G$ as in Table \ref{dual.coxeter}.  (In case $G$ has type $A_2$, $\alpha_{\Ad} = 1/4$.  For the other types, $\dim G + 2$ is of the form $2^x 3^y 5^z$ for some $x$, $y$, $z$.)  Rewriting a formula for $\dim G$ in terms of $\hv$ from \cite[p.~431]{CohendeMan} or the polynomial in \cite[3.17]{Okubo:quartic} produces the remarkable formula:
\begin{equation} \label{alpha.hv}
4  \alpha_{\Ad} (\hv)^2 = \hv + 6.
\end{equation}
(This is just one example from many families of formulas, compare for example \cite{Del}, \cite{DeligneGross}, \cite{LM:tri}, and \cite{LandsMan:univ}.)
\end{eg}

Here is the promised embedding.

\begin{prop} \label{pi.prop}
If $G$ has type $A_2$, $G_2$, $F_4$, $E_6$ or $E_7$ and $\pi \!: G \to \GL(V)$ is an irreducible representation of dimension $3$, $7$, $26$, $27$, or $56$ respectively, then formula
\begin{equation} \label{pi.sdef}
\s(S(XY)) = 6 \hv \pi(X) \jord \pi(Y) - \frac12 K(X,Y) \Id_B \quad \text{for $X \in \g$}.
\end{equation}
defines an injective $G$-equivariant linear map
\[
\s \!: A(\g) \hookrightarrow \End(V).
\]
If additionally $G$ is not of type $A_2$, then $\s$ satisfies
\begin{equation} \label{pi.proj}
\Proj_{\pi(\g)}(\s(S(X^2)) \jord \pi(Y)) = \pi \left( S(X^2) Y  \right) \quad \text{for $Y \in \g$}.
\end{equation}
\end{prop}

\begin{proof}
The representation $\pi$ is irreducible.  Moreover, one checks that in each case we have:
\[
\mu_\pi = \frac{\hv + 1}{\hv + 6},
\]
which by \eqref{alpha.hv} is the same as 
\begin{equation} \label{pi.hv}
\hv = \frac{2 + d}{2(6\mu_\pi - 1)}
= \frac{\mu_\pi d_\pi}{4 \alpha_\pi d},
\end{equation}
where we have 
abbreviated $d_\pi := \dim V$ and $d := \dim G$.

Recall that $\Sym^2 \g = k \oplus L(\la) \oplus L(2\hst)$ as a representation of $G$ whereas, at least in case $k = \C$,  $\End(V)$ and $A(\g)$ contain $k$ and $L(\la)$ with multiplicity 1 and do not contain $L(2\hst)$.  It follows that any $G$-equivariant linear map $\Sym^2 \g \to \End(V)$ factors through $S \!: \Sym^2 \g \to A(\g)$.  In particular, the map $XY \mapsto 6\hv \pi(X)\jord \pi(Y) - \frac12 K(X,Y) \Id_V$ does so, whence the map $\s$ from \eqref{pi.sdef} is well defined.  This $\s$ is defined over $R$, and so it is also well defined for $k$.

We now verify \eqref{pi.proj}, so assume $G$ is not of type $A_2$.  Linearizing \eqref{2.4a1} in $Y$ gives
\begin{align*}
\Tr((\pi(X)^2 \jord \pi(Y)) \pi(Z)) = &-\frac{\mu_\pi d_\pi}{6d} K([X,Y], [X,Z]) \\
&\ + \frac{2 \alpha_\pi}3 K(X,Y)K(X,Z) + \frac{\alpha_\pi}3 K(X,X) K(Y,Z),
\end{align*}
As $K(Y,Z) = \frac{d}{\mu_\pi d_\pi} \Tr(\pi(Y) \pi(Z))$ (Lemma \ref{qcas}), we have
\[
\frac{\alpha_\pi}3 K(X,X) K(Y,Z) = \Tr \left(\left(\frac{d \alpha_\pi}{3 \mu_\pi d_\pi} K(X,X) \Id_B \jord \pi(Y)\right) \pi(Z)\right).
\]
We obtain
\begin{multline*}
\Tr\left( \left( \left( \pi(X)^2 - \frac{d\alpha_\pi}{3 \mu_\pi d_\pi} K(X,X)\Id_B\right) \jord \pi(Y) \right) \pi(Z) \right) = \\
\frac{2\alpha_\pi}3 K \left( \left( \frac{\mu_\pi d_\pi}{4 \alpha_\pi d} (\ad X)^2 + P(X^2)\right) Y, Z \right).
\end{multline*}
Multiplying both sides by $6\hv$ and applying \eqref{pi.hv} gives \eqref{pi.proj}.
\end{proof}

\begin{eg}[$A(\sl_3)$] \label{A2}
The case $\g = \sl_3$ was included in Table \ref{dual.coxeter} but excluded from \S\ref{pa.sec}, so we now use the preceding construction to describe $A(\sl_3)$.  For $X, Y \in \sl_3$, $\Tr(XY) = \frac16 K(X,Y)$ by Lemma \ref{qcas}, so the embedding $\s\!: A(\sl_3) \to M_3(k)$ is via
\[
\s(S(X^2)) = 18X^2 - 3 \Tr(X^2) I
\]
and it is an isomorphism by dimension count.  We define a product $*$ on $M_3(k)$ via
$P * Q := \s^{-1}(P) \am \s^{-1}(Q)$.  Putting $\counit := \frac13 \Tr$ for the counit and chasing through the formulas, we find:
\begin{equation} \label{F}
P *Q = \left[  \textstyle\frac12 \counit(P \jord Q) - \textstyle\frac32 \counit(P) \counit(Q) \right] I + \counit(Q) P + \counit(P) Q.
\end{equation}
That is, $M_3(k)$ with the multiplication $*$ is of the form $\Unit(\sl_3, f)$ with notation as in the appendix, where the multiplication on $\sl_3$ is taken to be identically zero and $f(P, Q) = \textstyle\frac12 \counit(P \jord Q)$.  This is the Jordan algebra constructed from the bilinear form $f$ as in \cite[pp.~13, 14]{Jac:J}, cf.~Remark \ref{jenner.rmk}.
\end{eg}

 %%%%%%%%%%%%%%%%%%%%%%%%%%%%%%%%%%%%%%%%%%%%%%%%%%%%%%%%%

\section{Final remarks} \label{final.sec}

We have defined here a construction that takes a simple algebraic group $G$ (equivalently, a simple Lie algebra $\g$) over a field $k$, with mild hypotheses on the field $k$, and gives an explicit formula \eqref{prod.full} for the multiplication on a unital $k$-algebra $A(\g)$ on which $G$ acts by automorphisms.  We used the description of $A(\g)$ as a representation of $G$ to show that it is a simple algebra, that the bilinear form on it is nondegenerate, and that for $G$ of type $F_4$ or $E_8$ the automorphism group is exactly $G$.

\subsection*{Computation}
One can construct $A(\g)$ in a computer in a way amenable to computations as follows.  First, construct $G$ or $\g$ together with its adjoint representation or, in the cases where Proposition \ref{pi.prop} applies, its natural representation.  Pick a basis $\{ X_i \}$ of $\g$, and compute $S(X_i X_j) \in \End(\g)$ in the first case or $\s(S(X_i X_j)) \in \End(B)$ in the second, for $i \le j$.  Among these elements, select a maximal linearly independent subset; it is a basis for $A(\g)$.  For each pair of basis elements, one may calculate the product $\am$ using \eqref{prod.full}, and express the result in terms of the chosen basis.  This gives the ``structure constants'' for the algebra.  Magma \cite{Magma} code implementing this recipe can be found at \url{github.com/skipgaribaldi/chayet-garibaldi}.

\subsection*{Polynomial identities}
Among the algebras $A(\g)$ for $G$ in Deligne's exceptional series, the cases $A(\sl_2)$ and $A(\sl_3)$ are unusual for being Jordan algebras and in particular power-associative, whereas $A(\g)$ is not power-associative for other choices of $\g$ (Prop.~\ref{nonpa}\eqref{nonpa.nonpa}).  It is natural, then, to ask what identities $A(\g)$ does satisfy in those  cases.  It does not satisfy any polynomial identity of degree $\le 4$ that is not implied by commutativity (Prop.~\ref{deg4}).  Moreover, in the case $G = G_2$, we verified using a computer that $A(\g)$ and also $\Unit(V, cf)$ for every $c \ne 1$ do not satisfy any degree 5 identity not implied by commutativity, leveraging the classification of such identities from \cite[Th.~5]{Os:id}.

In case $G = G_2$ or $E_8$, the $G$-module $\Sym^2 V$ has only 6 summands, which suggests the existence of an identity of degree $\le 7$ in view of Example \ref{deg7}.  In the case of $G_2$, the 26 nonassociative and commutative monomials of degree $\le 7$ in an element $a \in A(\g_2)$ are linearly dependent.  We have found a ``weighted'' polynomial identity for $A(\g_2)$ in the sense of \cite{Tkachev:half}, i.e., for each nonassociative monomial $m$ of degree $\le 7$, there is a polynomial function $\phi_m$ on $A(\g_2)$ so that the  function $a \mapsto \sum_m \phi_m(a) m(a)$ is identically zero.  It would be interesting to know whether a similar idenitty holds for $A(\mathfrak{e}_8)$.  

 %%%%%%%%%%%%%%%%%%%%%%%%%%%%%%%%%%%%%%%%%%%%%%%%%%%%%%%%%
\appendix
\section{Adjoining a unit to a $k$-algebra} \label{unitizing.sec}

We carefully record in this appendix some details concerning adjoining a multiplicative identity to a $k$-algebra, because we do not know a sufficient reference for this material.
Suppose we are given a $k$-algebra $V$ that may not contain a multiplicative identity.   That is, $V$ is a vector space over $k$ together with a $k$-bilinear map $\cdot \!: V \times V \to V$, which we call the multiplication on $V$.  Given a bilinear form $f$ on $V$, we define a \emph{unital} $k$-algebra $\Unit(V, f)$ that has underlying vector space $k \oplus V$ and multiplication
\begin{equation} \label{Phi.mult}
(x_0, x_1)(y_0, y_1) = (x_0 y_0 + f(x_1, y_1), x_0 y_1 + y_0 x_1 + x_1 \cdot y_1)
\end{equation}
for $x_0, y_0 \in k$ and $x_1, y_1 \in V$.
Then $(1, 0)$ is the multiplicative identity in $\Unit(V, f)$ and $V$ is a subalgebra.

\begin{rmk}
The construction $\Unit(V, f)$ is discussed from a different point of view in Fox's paper \cite[\S5]{Fox:ca}.  A specific example of this construction in earlier literature comes from the 196883-dimensional Griess algebra $V$, whose automorphism group is the Monster.  Fox points out (Example 5.7) that various choices of $f$ are used in the literature when authors add a unit to $V$.
\end{rmk}

In the literature, one commonly finds the more restrictive recipe $\Unit(V, 0)$ for adjoining a unit to $V$ (i.e., where $f$ is identically zero), see for example \cite[Ch.~II]{Schfr}.  This has the advantage of not introducing the parameter $f$, however it has the disadvantage of always producing a non-simple algebra --- $V$ is an ideal in $\Unit(V, 0)$ --- and therefore it does not produce popular examples of simple algebras like the $n$-by-$n$ matrices over a field, the octonions, or Albert algebras.  For more on this, see Proposition \ref{unit.simple} below.

\begin{rmk} \label{unit.nomu}
One could imagine generalizing the construction to add a further parameter $\mu \in k$ and defining $\Unit(V, f, \mu)$ to have the same underlying vector space as $\Unit(V, f)$ but with multiplication rule 
\[
(x_0, x_1)(y_0, y_1) = (x_0 y_0 + f(x_1, y_1), x_0 y_1 + y_0 x_1 + \mu x_1 y_1).
\]
It is easily seen, however, that $\Unit(V, f, \mu)$ is isomorphic to $\Unit(V, \mu^{-2} f)$, so no generality would be gained.
\end{rmk}

Throughout the remainder of this section, \emph{we assume that all algebras considered are finite-dimensional}.

\subsection*{Counit}
For a $k$-algebra $A$ with multiplicative identity $e$, we call a $k$-linear map $\counit \!: A \to k$ such that $\counit(e) = 1$ a \emph{counit}.  Such a map gives a direct sum decomposition $A = ke \oplus V$ as vector spaces where $V := \ker \counit$ and furthermore expresses $A$ as an algebra $\Unit(V, f)$ by setting
\begin{equation} \label{AV}
f(v,v') := \counit(vv') \quad \text{and} \quad v \cdot v' := vv' - f(v,v') \quad \text{for $v, v' \in V$.}
\end{equation}
Conversely, every algebra $\Unit(V, f)$ has a natural counit, namely the projection of $k \oplus V$ on  its first factor.  In this way, we may identify the notions of unital $k$-algebras with a counit on the one hand and algebras of the form $\Unit(V, f)$ (with specified $V$ and $f$) on the other.

Additionally, a counit defines a bilinear form $\tbil$ on $A$ by setting 
\begin{equation} \label{tbil.def}
\tbil(a, a') := \counit(aa') \quad \text{for all $a, a' \in A$.}
\end{equation}
  Evidently, the direct sum decomposition $A = ke \oplus V$ is an orthogonal sum with respect to $\tbil$, i.e., $\tbil(e, v) = 0$ for all $v \in V$, and the restriction of $\tbil$ to $V$ is $f$.  From this it follows that $\tbil$ is symmetric (resp.~nondegenerate) if and only if $f$ is symmetric (resp.~nondegenerate).
  
\begin{eg} \label{counit.canonical}
In the special case where the integer $\dim A$ is not zero in $k$, there is a natural counit $\counit: a \mapsto \frac1{\dim A} \Tr(M_a)$, where we have written $M_a \in \End(A)$ for the linear transformation $b \mapsto ab$.   Therefore there is a natural way of writing $A$ as $\Unit(V, f)$ for $V$ and $f$ as in \eqref{AV}.  Moreover, every algebra automorphism of $A$ preserves $\counit$, whence the group scheme $\Aut(A)$ is identified with the sub-group-scheme of $\GL(V)$ of transformations that preserve both the multiplication $\cdot$ and the bilinear form $f$.
\end{eg}

Recall that a bilinear form on a $k$-algebra is called \emph{associative} if it satisfies \eqref{t.assoc}.

\begin{prop} \label{app.assoc}
In the notation of the preceding four paragraphs, $\tbil$ is associative   (with respect to the algebra $A$) if and only if $f$ is associative (with respect to the algebra $V$).
\end{prop}

\begin{proof}
Write elements $a, a', a'' \in A$ as $a = (a_0, a_1)$, etc.  Then $\tbil(aa', a'') - \tbil(a, a'a'') = f(a_1 \cdot a'_1, a''_1) - f(a_1, a'_1 \cdot a''_1)$.
\end{proof}

The property of being metrized, i.e., of having a nondegenerate and associative bilinear form, has the following interesting consequence.

\begin{prop} \label{deg4}
Let $A$ be a commutative $k$-algebra where $\car k \ne 2, 3, 5$ and suppose that $A$ is metrized.  If $A$ satisfies an identity of degree $\le 4$ not implied by commutativity, then $A$ satisfies the Jordan identity $x(x^2y) = x^2(xy)$ and is power-associative.
\end{prop}

\begin{proof}
  Writing $\top$ for the involution on $\End(A)$ corresponding to the nondegenerate associative bilinear form on $A$, we have $M_a^\top = M_a$ and $(M_a M_b)^\top = M_b M_a$ for all $a, b \in A$.  Note that the Jordan identity is equivalent to the assertion that $[M_a, M_{a^2}] = 0$ for all $a \in A$.

According to \cite[Th.~4]{Os:id}, $A$ satisfies \eqref{pa.1} or
\begin{enumerate}
\setcounter{enumi}{6}
\item \label{O7} $2((yx)x)x+ yx^3 = 3(yx^2)x$ or
\item \label{O8} $2(y^2x)x-2((yx)y)x - 2((yx)x)y + 2(x^2y)y - y^2x^2+(yx)^2 = 0$.
\end{enumerate}
Identity \eqref{O7} is equivalent to the statement $2M_x^3 + M_{x^3} = 3M_x M_{x^2}$.  Applying $\top$ to this identity, subtracting it, and dividing by 3, we obtain $[M_x, M_{x^2}] = 0$.

For \eqref{pa.1}, replacing $a$ with $x+y$, expanding, and taking the terms of degree 1 in $y$, we find $M_{x^3} + M_x M_{x^2} + 2 M_x^3 = 4M_{x^2} M_x$.  Applying $\top$ to this identity, subtracting it, and dividing by 5 gives $[M_x, M_{x^2}] = 0$.

Finally, if \eqref{O8} holds, then replacing $y$ with $y+z$ and taking the terms of degree 1 in $y$ and $z$, replacing $y$ with $x$ , and applying the same procedure as in previous cases again gives $[M_x, M_{x^2}] = 0$. 
\end{proof}

For comparison, the situation when $A$ is not assumed to be metrized is more complicated, see \cite{Os:deg4} and \cite{CariniHC}.

The following example provides a positive statement.

\begin{eg} \label{deg7}
Let $A$ be a commutative $k$-algebra that is metrized, and suppose that the $\Aut(A)$-module $\Sym^2 A$ has a composition series of length $d$.
Define $P_e \!: \Sym^e A \to \End(A)$ via
\[
P(a_1 a_2 \cdots a_e) := \sum_{\text{permutations $\sigma$}} M_{a_{\sigma(1)}} M_{a_{\sigma(2)}} \cdots M_{a_{\sigma(e)}}
\]
This is $\Aut(A)$-equivariant and its image $H_e$ is contained in the space of symmetric operators on $A$ with respect to $\tbil$, which we identify with $\Sym^2 A$.  Setting $H_0 := k \Id_A$ and $I_e := H_0 + H_1 + \cdots + H_e$, we obtain an increasing chain of submodules $0 \ne I_0 \subsetneq I_1 \subseteq \cdots$ so that $I_e =I_{e+1}$ for some $e < d$.  That is, a symmetric expression
\[
\sum_\sigma a_{\s(1)}(a_{\s(2)}(a_{\s(3)} \cdots (a_{\s(e+1)} b))\cdots ) \quad \in A,
\]
where each summand is a product of at most $d + 1$ terms,
can be expressed in terms of symmetric expressions in the $a$'s involving products of fewer terms.  
\end{eg}

\subsection*{Simplicity} A $k$-algebra $A$ is \emph{simple} if the only two-sided ideals in $A$ are $0$ and $A$ itself.  We prove the following criterion for simplicity.  

\begin{prop} \label{unit.simple}
Let $A$ be a unital $k$-algebra with counit $\counit$.  If
\begin{enumerate}
\item \label{unit.connected} there is a connected group scheme $G \subseteq \Aut(A)$ that stabilizes $\counit$;
\item \label{unit.KS} $k$ is not a composition factor of $\ker \counit$ as a $G$-module; and
\item $\tau$ as defined in \eqref{tbil.def} is nondegenerate,
\end{enumerate}
then $A$ is simple.
\end{prop}

\begin{rmk} \label{jenner.rmk}
In the case where the multiplication on $V := \ker \counit$ is identically zero, the algebra $A$ is of the kind studied in \cite{Jenner}.
\end{rmk}

\begin{proof}[Proof of Proposition \ref{unit.simple}]
Put $V := \ker \counit$.  We first claim that every $G$-invariant subspace $I$ of $A$ is a direct sum $I = (ke \cap I) \oplus (V \cap I)$.  If the restriction of the projection $1 - \counit: A \to V$ to $I$ has a kernel, then $\ker (1 - \counit) = ke$ is contained in $I$ and the claim is clear.  Otherwise, $1 - \counit$ is injective and $I = \{ (\pi(w), w) \}$ for $w \in W := (1 - \counit)(I)$ and some $G$-equivariant linear map $\pi \!: W \to k$.  By \eqref{unit.KS}, however, $\pi$ must be zero, and the claim follows.

We next verify that every nonzero and $G$-invariant ideal $I$ of $A$ is equal to $A$.  By the preceding paragraph, we may suppose that there is a nonzero $v \in V \cap I$.  Since $\tbil$ is nondegenerate, there is an $a \in A$ so that $0 \ne \tau(v, a) = \counit(va)$.  That is, $va$ is a nonzero element of $ke \cap I$, whence $I = A$.

Now let $I$ be a nonzero ideal in $A$.  The sum of $G$-conjugates of $I$, $\sum_g gI$ is a nonzero and $G$-invariant ideal, so it equals $A$.  We conclude that $I$ itself equals $A$ by arguing as in the proof of \cite[Th.~5]{Popov:artin}, which concerns the analogous case of a non-unital algebra that is an irreducible representation of a connected group.
\end{proof}

\subsection*{Power-associativity}
A $k$-algebra $A$ is \emph{power-associative} if the subalgebra generated by any element $a \in A$ is associative.  It is \emph{strictly power-associative} if $A \ot_k F$ is power-associative for every field $F$ containing $k$.  We now focus on the case where $A$ is commutative, as is the algebra $A(\g)$ elsewhere in this paper and as is the algebra $\Unit(V, f)$ when $V$ is commutative and $f$ is symmetric.  

If $A$ is power-associative, then in particular
\begin{equation} \label{pa.1}
a(a(aa)) - (aa)(aa) = 0 \quad \text{for all $a \in A$.}
\end{equation}
When $\car k \ne 2, 3, 5$, \eqref{pa.1} is equivalent to $A$ being strictly power-associative \cite[Th.~1]{Alb:pa}, cf.~\cite[p.~364]{Kokoris}.

The property of whether $\Unit(V, f)$ is strictly power-associative is rather constrained.    In the proposition below, we write $v^2$ for the element $v \cdot v \in V$.

\begin{prop} \label{unit.unique}
Suppose $f$ is not alternating.  If the polynomial map $v \mapsto v \wedge v^2 \in \wedge^2 V$ is not identically zero, then there is at most one $c \in k$ so that $\Unit(V, cf)$ is strictly power-associative.
\end{prop}

\begin{proof}
We focus on \eqref{pa.1} for $a \in \Unit(V, cf)$.  Writing out $a = (a_0, a_1)$ and expanding $a(a(aa)) - (aa)(aa)$, we find $(c(f(a_1, a_1 ( a_1^2)) - f(a_1^2, a_1^2)), x + cy)$ for 
\begin{equation} \label{pa.2}
x = a_1(a_1a_1^2) - a_1^2 a_1^2 \quad \text{and} \quad
y = f(a_1, a_1^2) a_1 - f(a_1, a_1) a_1^2.
\end{equation}
By hypothesis, $a_1$, $a_1^2$ are linearly independent for generic $a_1 \in V$.  And $f(a_1, a_1)$ is also nonzero for generic $a_1 \in V$ because $f$ is not alternating, so we conclude that $y$ is not the zero polynomial on $V$.  It follows that the polynomial function $x + cy$ on $V$ is identically zero for at most one value of $c \in k$.
\end{proof}

\begin{rmk} 
For $a = (a_0, a_1) \in \Unit(V,f)$, we have $(1,0) \wedge a \wedge a^2 = (1,0) \wedge (0,a_1) \wedge (0, a_1^2)$ in $\wedge^3 \Unit(V, f)$, where the squaring operation on the left side is relative to the multiplication on $\Unit(V, f)$ and on the right side is relative to the multiplication $\cdot$ on $V$.  As a consequence, the hypothesis of Proposition \ref{unit.unique} can be phrased in the equivalent form: \emph{the polynomial map $a \mapsto (1,0) \wedge a \wedge a^2$ is not zero.}
\end{rmk}

\begin{rmk} \label{octonions.rmk}
If the squaring map $v \mapsto v^2$ is the zero function, the identities $a^2 a = a a^2$ and \eqref{pa.1} hold in $\Unit(V, cf)$.  If $\car k = 0$, it follows that $\Unit(V, cf)$ is strictly power-associative for every $c \in k$ by \cite[Th.~2]{Alb:pa}.  
\end{rmk}

The following lemma allows one to apply Proposition \ref{unit.unique} in situations such as that in Proposition \ref{unit.simple}, by taking $F(v) = v^2$.
\newcommand{\Fb}{\overline{F}}
\begin{lem}\label{lin.ind}
Suppose $\dim V \ge 2$ and let $F$ be a $G$-equivariant polynomial function $V \to V$ that is homogeneous of degree $d \ge 1$.  If the polynomial map $v \mapsto v \wedge F(v)$ is identically zero, then  there is a $G$-invariant polynomial function $\Fb \!: V \to k$ that is homogeneous of degree $d - 1$ and $F(v) = \Fb(v)v$ for all $v \in V \otimes K$ for every extension $K$ of $k$.
\end{lem}

\begin{proof}
There is a $G$-invariant function $\Fb \!: V \setminus \{ 0 \} \to k$ defined implicitly by the equation $\Fb(v) v = F(v)$.  We argue that it is a polynomial function on $V$.

Fix a basis $x_1, \ldots, x_n$ of $V^*$.  The $i$-th coordinate $x_i\vert_{F(v)}$ of $F(v)$ is $f_i(v)$ for some homogeneous degree $d$ polynomial $f_i \in k[x_1, \ldots, x_n]$. 
On the open set $U_i$ where $x_i$ does not vanish, $\Fb = f_i/x_i$.  For $i \ne j$, $f_i/x_i$ and $f_j/x_j$ agree on $U_i \cap U_j$, so $x_i f_j = x_j f_i$ in the polynomial ring.  As $x_i$ does not divide $x_j$, it must divide $f_i$.  Setting $\bar{f}_i := f_i / x_i$, the polynomial function  $v \mapsto F(v) - \bar{f}_i(v) v$ is zero on $U_i$, so it is zero on $V$, i.e., $\Fb \!: V \to k$ is a polynomial.
\end{proof}

%\bibliographystyle{amsalpha}
%\bibliography{bibfile}

\providecommand{\bysame}{\leavevmode\hbox to3em{\hrulefill}\thinspace}
\providecommand{\MR}{\relax\ifhmode\unskip\space\fi MR }
% \MRhref is called by the amsart/book/proc definition of \MR.
\providecommand{\MRhref}[2]{%
  \href{http://www.ams.org/mathscinet-getitem?mr=#1}{#2}
}
\providecommand{\href}[2]{#2}

\end{document}